\documentclass[10pt]{amsart}
\usepackage{latexsym, amsmath,amssymb}
\usepackage{url}
  \usepackage{graphicx}

\renewcommand{\epsilon}{\varepsilon}
	\newcommand{\newsection}[1]
{\subsection{#1}\setcounter{theorem}{0} \setcounter{equation}{0}
\par\noindent}
\newtheorem{theorem}{Theorem}

\newtheorem{lemma}[theorem]{Lemma}
\newtheorem{corr}[theorem]{Corollary}

\newtheorem{proposition}[theorem]{Proposition}
\newtheorem{deff}[theorem]{Definition}

\newcommand{\bth}{\begin{theorem}}
\newcommand{\ble}{\begin{lemma}}
\newcommand{\bcor}{\begin{corr}}

\newcommand{\bdeff}{\begin{deff}}

\newcommand{\bprop}{\begin{proposition}}
\newcommand{\ele}{\end{lemma}}
\newcommand{\ecor}{\end{corr}}
\newcommand{\edeff}{\end{deff}}

\newcommand{\eprop}{\end{proposition}}

\newcommand{\cd}{\, \cdot\, }

\newcommand{\Rn}{{\mathbb R}^n}

\newcommand{\la}{\lambda}

\newcommand{\e}{\varepsilon}

\renewcommand{\Pi}{\varPi}
\renewcommand{\Re}{\rm{Re} \,}
\renewcommand{\Im}{\rm{Im} \,}
\renewcommand{\epsilon}{\varepsilon}
\newcommand{\sgn}{{\mathrm{sgn}}}

\newcommand{\Rt}{{\Bbb R}^3}

\newcommand{\R}{{\mathbb R}}
\newcommand{\C}{{\mathbb C}}

\newcommand{\Csl}{\C \backslash \overline{\R_+}}
 \newcommand{\Rzeta}{\mathfrak{S}_{loc}(\la,\mu)}
\newcommand{\Zn}{{\mathbb Z}^n}

\newcommand{\Lleft}{L^{\frac{2n}{n-2}}(M)}
\newcommand{\Lright}{L^{\frac{2n}{n+2}}(M)}

	\newcommand{\Rlm}{\mathfrak{S}_{loc}(\la,\mu)}

\newcommand{\IRn}{\int_{\Rn}}
\newcommand{\Spec}{\text{Spec}}
\newcommand{\Tn}{{\mathbb T}^n}

\newcommand{\Lnorm}{L^{\frac{2n}{n-2}}}
\newcommand{\Rnorm}{L^{\frac{2n}{n+2}}}

\newcommand{\Lspace}{L^{\frac{2n}{n-2}}}
\newcommand{\Rspace}{L^{\frac{2n}{n+2}}}

\newcommand{\So}{\mathfrak{S}}

\newcommand{\Poi}{e^{i(\text{sgn}\, \mu) \lambda t}e^{-|\mu|t}\,}
\newcommand{\Poie}{e^{i(\text{sgn}\,  \mu) \lambda\e t}e^{-|\mu|\e t}\,}

\newcommand{\Gt}{\sqrt{-\Delta_g}}

\newcommand{\Ct}{\sqrt{-\Delta_{\tilde g}}}

\def\nint{\not\!\!\int}
\newcommand{\otau}{\overset o\tau}

\newcommand{\subheading}[1]{{\bf #1}}

\begin{document}

\subjclass[2010]{Primary, 58J50; Secondary 35R01, 42C99}
\keywords{Resolvent estimates,  eigenfunctions, spectrum, curvature}

\title[$L^p$-resolvent estimates on compact Riemannian manifolds]
%{Towards optimal $L^p$-resolvent estimates \\ on compact Riemannian manifolds}
{On $L^p$-resolvent estimates and 
the density of eigenvalues
for compact
Riemannian 
manifolds}
\thanks{The second and third authors were supported in part by the NSF grant DMS-1069175. 
The research was carried out while
the fourth author was visiting Johns Hopkins University, supported by the 
%China Scholarship Council.
Program for New Century Excellent Talents in University (NCET-10-0431).
}
%\thanks{The authors are grateful for the numerous suggestions of the referees which improved the exposition of the paper.}

\author{Jean Bourgain}
\address{Institute for Advanced Study, Princeton, NJ}
\email{bourgain@ias.edu}
\author{Peng Shao}
\address{Department of Mathematics,  Johns Hopkins University,
Baltimore, MD}
\email{pshao@jhu.edu}
\author{Christopher D. Sogge}
\address{Department of Mathematics,  Johns Hopkins University,
Baltimore, MD}
\email{sogge@jhu.edu}
\author{Xiaohua Yao}
\address{Department of Mathematics, Huazhong Normal University, Wuhan 430079, PR China}
\email{yaoxiaohua@mail.ccnu.edu.cn}

\begin{abstract}
We address an interesting question raised by Dos Santos Ferreira, Kenig and Salo~\cite{Kenig} about regions ${\mathcal R}_g\subset {\mathbb C}$
for which there can be uniform $L^{\frac{2n}{n+2}}\to L^{\frac{2n}{n-2}}$ resolvent estimates for $\Delta_g+\zeta$, $\zeta \in {\mathcal R}_g$, where $\Delta_g$ is the Laplace-Beltrami operator with metric $g$ on a given compact boundaryless
Riemannian manifold of dimension $n\ge3$.  This is related to  earlier work of Kenig, Ruiz and the third author~\cite{KRS} for the
Euclidean Laplacian, in which case the region is the entire complex plane minus any disc centered at the origin.   Presently, we show that for the round metric on the sphere, $S^n$,
the resolvent estimates in \cite{Kenig}, involving a much smaller region, are essentially optimal.  
We do this  by establishing sharp bounds based on the distance from $\zeta$ to the spectrum of $\Delta_{S^n}$.
 In the other direction, we also  show that the bounds in \cite{Kenig} can be sharpened logarithmically  for manifolds with nonpositive curvature, and   by powers in the case of the torus, ${\mathbb T}^n={\mathbb R}^n/{\mathbb Z}^n$, with the flat metric.  The latter improves earlier bounds of Shen~\cite{Shen}.
The work of \cite{Kenig} and \cite{Shen} was based on Hadamard parametrices for $(\Delta_g+\zeta)^{-1}$.  Ours is based on the related Hadamard parametrices for
$\cos t\sqrt{-\Delta_g}$, and it follows  ideas in  \cite{Sogge2} of proving $L^p$-multiplier estimates using small-time wave equation parametrices and the spectral projection estimates
from \cite{Sogge1}.  This approach allows us to adapt arguments in B\'erard~\cite{Berard} 
and Hlawka~\cite{Hlawka}
to obtain  the aforementioned improvements over
\cite{Kenig} and \cite{Shen}.  Further improvements for the torus are obtained using recent techniques of the first author \cite{B1} and
his work with Guth \cite{B-G} based on the multilinear estimates of Bennett, Carbery and Tao \cite{B-C-T}.
Our approach also allows us to give a natural necessary condition for favorable resolvent estimates that is based on a 
measurement of the density of the spectrum of $\sqrt{-\Delta_g}$, and, moreover, a necessary and sufficient
condition based on natural improved spectral projection estimates for shrinking intervals, as opposed
to those in \cite{Sogge1} for unit-length intervals. We show that the resolvent estimates are sensitive to clustering within the spectrum, which is not surprising given Sommerfeld's original
conjecture \cite{Som} about these operators.
\end{abstract}

\maketitle

\newsection{Introduction}

The purpose of this paper is to address a question of Dos Santos Ferreira, Kenig and Salo~\cite{Kenig} about the regions ${\mathcal R}_g\subset
\C$ for which there can be uniform $L^p$-resolvent bounds of the form
\begin{equation}\label{1.1}
\|u\|_{\Lleft}\le C_{{\mathcal R}}\|(\Delta_g+\zeta)u\|_{\Lright}, \quad \zeta \in {\mathcal R}_g,
\end{equation}
if $\Delta_g$ is the Laplace-Beltrami operator on a compact boundaryless Riemannian manifold $(M,g)$ of dimension $n\ge 3$.  We shall
be able to obtain sharp results in some cases (Zoll manifolds) and improvements over the known results in others (nonpositive curvature).  The
results that we obtain are related to known bounds for the remainder term in the sharp Weyl formula for $\Delta_g$, which measures 
how uniformly its spectrum is distributed.  It is natural to consider the pair of exponents $(2n/(n+2),2n/(n-2)$, since these are
the ones on the line of duality that occur for the Euclidean Laplacian in $n$-dimensions.

Our results also are related to Sommerfeld's original conjecture
and reasoning regarding the resolvent  operators that are associated with \eqref{1.1}, \cite{Som} (see \cite{A}).  
Recall that solutions of the Helmholtz equation,
$$\Delta_gu(x) + \la^2u(x) =F(x),$$
give rise to solutions $w(t,x)=  u(x) \sin t\la$ of the dynamic equation of forced vibration,  $(\partial_t^2-\Delta_g)w(t,x)=\sin \la t  \, F(x)$.  A classical
problem proposed by Sommerfeld is to determine how solutions of the forced vibration equation are related to solutions
of the stationary free vibration equation $(\Delta_g+\la_j^2)e_j(x)=0$ (eigenfunctions) and the possible $\la_j$ (the frequencies).  
Theorems \ref{theorem1.2} and \ref{theoremnecsuff} below address this issue.
We are also able
to give a complete answer to the current variant, \eqref{1.1}, of this problem for the standard sphere, and an essentially sharp one
for Zoll manifolds, due to the tight clustering of the eigenvalues in these cases.
Also, in many ways, the pointwise estimates that we employ in the proof of our estimates,
as well as the negative results that we obtain, are in accordance with Sommerfeld's reasoning.  Note also how we have introduced
a square, $\zeta=\la^2$, into the problem, which will turn out to be a matter of bookkeeping that will simplify the analysis.
 
Regarding a related problem for Euclidean space, involving the standard Laplacian, $\Delta_{\Rn}$, on $\Rn$, $n\ge 3$, it was shown by Kenig, Ruiz and the third author (KRS) \cite{KRS} that
for each $\delta>0$
one has the uniform estimates
\begin{equation}\label{1.2}
\|v\|_{\Lnorm(\Rn)}\le C_\delta \|(\Delta_{\Rn}+\zeta)v\|_{\Rnorm(\Rn)}, \quad \text{if } \, \, \zeta \in \C, \, \, |\zeta|\ge \delta, \quad
\text{and } \, \, v\in {\mathcal S}(\Rn).
\end{equation}
In particular, the bound even holds when $\zeta$ is in the spectrum of $-\Delta_{\Rn}$, but of course \eqref{1.1} cannot hold if
${\mathcal R}$ intersects the spectrum of $-\Delta_g$, $\Spec(-\Delta_g)$, since the latter is discrete.

For the manifold case, the interesting question is how close ${\mathcal R}$ can come to $\Spec(-\Delta_g)$ near infinity and
still  have \eqref{1.1}.   Given $\delta>0$ it is very easy to see (see \S 2) that \eqref{1.1} is valid if ${\mathcal R}$ is
\begin{equation}\label{1.3}
{\mathcal R}^-_\delta=\{\zeta \in \C: \, |\zeta|\ge \delta \, \, \, \text{and } \, \, \text{Re }\zeta \le \delta\}.
\end{equation}
This essentially follows from elliptic regularity estimates, and one can prove the assertion (see \S 2) using the spectral
projection estimates of the third author \cite{Sogge1}.  

A very nontrivial result due to Shen \cite{Shen} (see also \cite{ShenZhao}) for the flat torus, $\Tn=\Rn/{\mathbb Z}^n$, $n\ge 3$,
and DKS \cite{Kenig} for general compact manifolds of these dimensions is that it also holds when ${\mathcal R}$ is
\begin{equation}\label{1.4}
{\mathcal R}_{DKSS}=\{ \zeta \in \C: \, \, (\text{Im} \, \zeta)^2\ge \delta\,  \text{Re }\zeta, \, \, \, \text{Re } \, \zeta \ge \delta\}
\cup {\mathcal R}^-_\delta.
\end{equation}
The boundary of the nontrivial part of this region is the curve $\gamma_{DKSS}$ in Figure 1 below.

In part following \cite{Sogge1}, Dos Santos Ferreira, Kenig and Salo \cite{Kenig} used the Hadamard parametrix for
$(\Delta_g+\zeta)^{-1}$ and Stein's \cite{Stein} oscillatory integral theorem.  In his work for the torus Shen~\cite{Shen} was
also able to make use of identities akin to the Poisson summation formula.  Once the parametrix was established
the main estimates were similar to the earlier ones in \eqref{1.2} for the Euclidean case of KRS \cite{KRS}.

In addition to proving the aforementioned results Dos Santos Ferreira, Kenig and Salo in \cite{Kenig} also 
asked whether \eqref{1.1} can hold for larger regions than the one in \eqref{1.4}, and in particular if ${\mathcal R}$ is
the region
\begin{equation}\label{1.5}
{\mathcal R}_{opt}=
\{ \zeta\in \C: \, |\text{Im} \, \zeta|\ge \delta \} \cup  {\mathcal R}^-_\delta.
\end{equation}
Clearly the condition that $|\text{Im }\zeta|$ be bounded below is needed since $\Delta_g$ has discrete spectrum.  This region
is the one bounded by the curve $\gamma_{opt}$ in the figure below.

\begin{figure}[placement h]
\begin{center}
\includegraphics[scale=0.3]{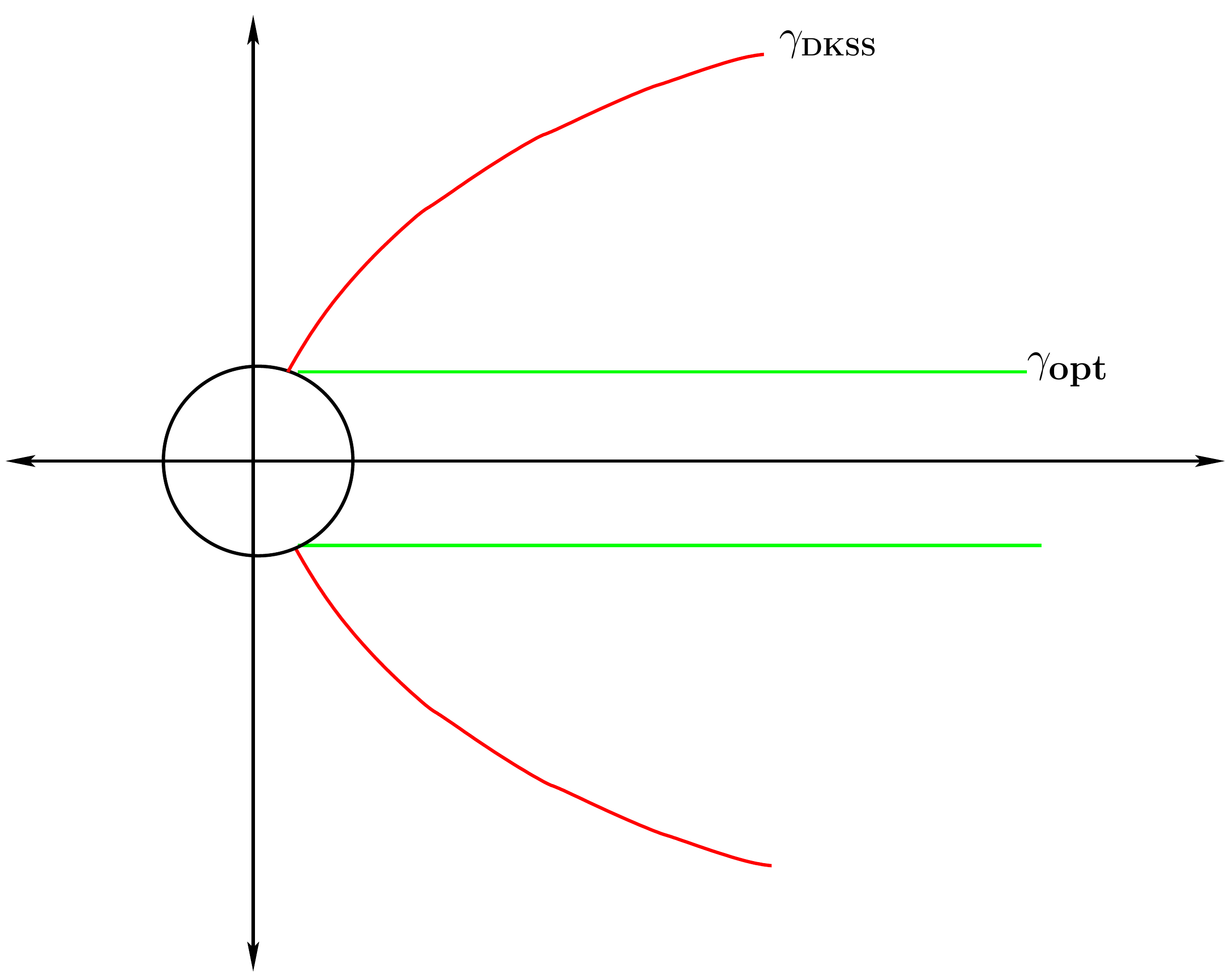}
\end{center}
\caption{Earlier results and the problem}
\end{figure}

To motivate our improvements over previous results and our progress on this question, let us return to the analogous bounds
\eqref{1.2} for the Euclidean case.  Taking $\zeta =1+i\varepsilon$, $0<\varepsilon \le 1$, we see that this estimate implies that
the multiplier operators
$$f\to (2\pi)^{-n}\int_{\Rn}e^{ix\cdot \xi}\bigl(|\xi|^2-1+i\varepsilon\bigr)^{-1} \, \Hat f(\xi)\, d\xi$$
are uniformly bounded from $\Rnorm(\Rn)$ to $\Lnorm(\Rn)$.   Considering the imaginary part of this operator,
one can see (cf. \cite{KRS}, p. 342) that this implies the following variant of a special case of the
Stein-Tomas restriction theorem \cite{Tomas} (written in a ``$TT^*$'' fashion):
\begin{multline}\label{1.6}
\Bigl\| (2\pi)^{-n}\int_{\{\xi\in \Rn: \, |\, |\xi|-1\, |\le \e\}} e^{ix\cdot \xi} \Hat f(\xi)\, d\xi\, \Bigr\|_{\Lleft(\Rn)}
\\
\le C\varepsilon \|f\|_{\Lright(\Rn)}, \quad 0<\varepsilon\le 1.
\end{multline}

The issue of whether \eqref{1.1} can hold for regions larger than the one in \eqref{1.4} is tied closely to what extent
bounds like this one are possible on  a given $(M,g)$, which is an interesting question in its own right.  To state things
more clearly, we need to introduce some notation.  First, we label and count with multiplicity the eigenvalues of
$$P=\sqrt{-\Delta_g}$$
as follows
$$0=\lambda_0\le \lambda_1\le \lambda_2\le \cdots .$$
We associate a real orthonormal basis of eigenfunctions $\{e_j(x)\}$, i.e., $-\Delta_ge_j(x)=\lambda^2_je_j(x)$, and let
$E_j$ denote projection onto the $j$th eigenspace:
$$E_jf(x)=\bigl(\int_M f(y)\, e_j(y)\, dV_g\, \bigr)\, e_j(x).$$
Then for a given $0<\e \le 1$, we consider the $\e$-scale spectral projection operators
\begin{equation}\label{1.7}
\chi_{[\la-\e,\la+\e]}f=\sum_{|\la_j -\la|\le \e} E_jf.
\end{equation}
From  \cite{Sogge1}, it follows that for $\e=1$, we have
\begin{multline}\label{1.8}
\bigl\|\chi_{[\la-1,\la+1]}f\bigr\|_{L^p(M)}\le C\la^{\sigma(p)}\|f\|_{L^{\frac{p}{p-1}}(M)},
\\
\text{if } \, p\ge \frac{2(n+1)}{n-1}, \, \, \, \, \text{and } \, \, \sigma(p)=2n\bigl(\frac12 - \frac1p\bigr)-1.
\end{multline}
The relevant issue for improving the resolvent bounds on a given Riemannian manifold $(M,g)$ is whether there is 
a function $\e(\la)$ taking on values in $(0,1]$ that decreases to $0$ as $\la\to +\infty$, i.e., $\e(\la)=o(1)$, for which
\begin{equation}\label{1.9}
\bigl\|\chi_{[\la-\e(\la),\la+\e(\la)]}f\bigr\|_{L^p(M)}\le C_p\, \e(\la) \la^{\sigma(p)}\|f\|_{L^{\frac{p}{p-1}}(M)},
\end{equation}
with a uniform constant $C_p$ for some $p\ge \frac{2(n+1)}{n-1}$.  The relevant case for improvements of \eqref{1.1}
beyond \eqref{1.4} of course is $p=\frac{2n}{n-2}$, but we are stating things this way to help motivate negative
results as well.

The latter are based on the fact that, as is well known, if \eqref{1.9} holds for a given exponent $p_0\in[\tfrac{2(n+1)}{n-1}, \infty)$,
then it must hold for all larger exponents, including $p=\infty$.  This assertion follows from Sobolev-type inequalities for
finite exponents and a Bernstein-type inequality for $p=\infty$ (see the remark after Lemma~\ref{lemma3.1} below).  The estimate
\eqref{1.9} for $p=\infty$ is equivalent to the statement that the kernel of the operator satisfies the following bounds
along the diagonal:
$$\chi_{[\la-\e(\la), \la+\e(\la)]}(x,x)=\sum_{|\la_j-\la|\le \e(\la)}|e_j(x)|^2 \le C\e(\la)\la^{n-1}.$$
As a result, the trace of the operator then would have to satisfy the following bounds if \eqref{1.9} were valid for some exponent
$\frac{2(n+1)}{n-1}\le p\le \infty$,
\begin{equation}\label{1.10}
N(\la+\e(\la))-N((\la-\e(\la))-)= \int_M \chi_{[\la-\e(\la), \la+\e(\la)]}(x,x) \, dV_g\le C\e(\la)\la^{n-1},
\end{equation}
with 
$$N(\la)=\#\{j: \, \la_j\le \la\}$$
being the Weyl counting function.  Here $N(\la-) = \lim_{\tau \nearrow \lambda}N(\tau)$, and thus the quantity in \eqref{1.10}
is $\#\{j: \, \la-\e(\la)\le \la_j\le \la+\e(\la)\}$.    Consequently, a necessary condition for \eqref{1.9} is that the number of eigenvalues in
bands of width $\e(\la)$ about $\la$ should be comparable to the size of the corresponding $\e(\la)$-annulus about the sphere of
radius $\lambda$ in $\Rn$.

No such results are possible on the sphere or Zoll manifolds,\footnote{Recall that a Zoll manifold is one for which the geodesic flow is periodic with a common minimal period $\ell$.  These manifolds are also sometimes called $P_\ell$ manifolds for this reason.  See \cite{Besse}.}
and so \eqref{1.9} cannot hold in this case with $\e(\la)\to 0$.  Correspondingly, the earlier resolvent bounds cannot be improved in this case:

\begin{theorem}\label{theorem1.1}  Let $(M,g)$ be a Zoll manifold of
dimension $n\ge 3$.  Then if 
$$\R_+\ni \tau\to \tau^2 + i \e(\tau)\tau=\zeta(\tau) \in {\mathbb C}$$
 is a curve
for which $\e(\tau)>0$ for all $\tau$ and $\e(\tau)\to 0$ as $\tau\to +\infty$, it follows that
\begin{equation}\label{1.11}
\sup_{1 \le\tau \le\la}\Bigl\|\, \bigl(\Delta_g+\zeta(\tau)\bigr)^{-1}\, \Bigr\|_{L^{\frac{2n}{n+2}}(M)
\longrightarrow L^{\frac{2n}{n-2}}(M)}\longrightarrow +\infty, \quad \text{as } \, \la \to +\infty.
\end{equation}
Moreover, for the Laplacian on the standard sphere, we have
\begin{multline}\label{1.12}
\bigl\|\, (\Delta_{S^n}+\zeta)^{-1}\, \bigr\|_{L^{\frac{2n}{n+2}}(S^n)\longrightarrow L^{\frac{2n}{n-2}}(S^n)}
\\
 \approx \max\bigl( \mathrm{dist }\bigl(\sqrt{\zeta}, \, 
\mathrm{Spec }\sqrt{-\Delta_{S^n}}\bigr)^{-1},  \, 1\bigr)
 \quad \mathrm{for } \, \, \zeta\in {\mathbb C}, \,  \mathrm{Re }\,  \zeta \ge 1, \, \,
 \text{and } (\Im \zeta)^2\le |\Re \zeta|.
\end{multline}
These bounds also hold for any Zoll manifold $(M,g)$ in the subset of this region where 
%$\sqrt{|\mathrm{Im}\zeta |}\ge C_M/\mathrm{Re}\, \zeta$, 
%$\mathrm{Re } \, \zeta\ge 1$,
$|\Im \sqrt\zeta |\ge C_M/ \Re \sqrt\zeta$,
with $C_M$ depending on $M$.
\end{theorem}

We are able to make these precise estimates on the norms in \eqref{1.12} since we know how the eigenvalues are distributed
with a great deal of precision.  In the case of the standard sphere, $S^n$, the distinct eigenvalues of $-\Delta_{S^n}$ are
$k(k+n-1)$ repeating with multiplicity $\approx k^{n-1}$.  If $(M,g)$ is a Zoll manifold with $2\pi$-periodic geodesic flow, 
then, by a theorem of Weinstein~\cite{Weinstein}, there is a 
number $\alpha=\alpha(M)$ so that, for large $k$, there are $\approx k^{n-1}$ eigenvalues of $-\Delta_g$ counted with multiplicity
in the intervals $[(k+\alpha)^2-C_M, (k+\alpha)^2+C_M]$, $k=1,2,\dots$.  The tight clustering of the spectrum is what accounts
for \eqref{1.12}.

Although we are only able to compute the resolvent norms in these special cases, we are able to obtain the following result which
gives a necessary condition to improvements of the earlier results in terms of the density of eigenvalues on small intervals.

\begin{theorem}\label{theorem1.2}
Suppose that $0\le \tau_k\to +\infty$ as $k\to \infty$ and suppose further that there exist $\e(\tau_k)>0$ satisfying
$\e(\tau_k)\searrow 0$ as $k\to \infty$ and
$$\bigl(\e(\tau_k)\tau_k^{n-1}\bigr)^{-1} \bigl[N(\tau_k+\e(\tau_k))-N(\tau_k-\e(\tau_k)-)\bigr]\longrightarrow +\infty.$$
Then
$$\Bigl\|\bigl(\Delta_g+\tau_k^2+i\tau_k\e(\tau_k)\bigr)^{-1}\Bigr\|_{L^{\frac{2n}{n+2}}(M)\longrightarrow L^{\frac{2n}{n-2}}(M)}
\longrightarrow +\infty.$$
\end{theorem}

Based on this, for there to be positive results for the problem posed in \cite{Kenig} for a given $(M,g)$, one would need
that $N(\la+1/\la)-N(\la-1/\la)=O(\la^{n-2})$, as $\la\to +\infty$, which appears to be an almost impossibly strong condition
(cf. \cite{Jac}, \cite{Szego}).  This is because
the curve $\gamma_{opt}$, which is the boundary of the region in the problem raised in \cite{Kenig} corresponds to $\e(\la)=c/\la$,

In the other direction, there are many cases where for $p=\infty$ \eqref{1.9} and hence \eqref{1.10} are valid.  The first place 
where this occurs for \eqref{1.9} seems to be in a paper of Toth, Zelditch and the third author~\cite{STZ}, who following
earlier work of Zelditch and this author~\cite{SZ1} showed that for {\em any} $M$ we have \eqref{1.9} with $\e(\la)=\e_g(\la)\to 0$
for a {\em generic} choice of metrics.  (See (1.8) in \cite{STZ} and \S 6 in \cite{SZ1}.)

One of our main results says that improvements for \eqref{1.9} for $p=\frac{2n}{n-2}$ provide a necessary and sufficient condition
for improvements of the earlier resolvent estimates:

\begin{theorem}\label{theoremnecsuff}  Let $(M,g)$ be a compact Riemannian manifold of dimension $n\ge3$.
Suppose that $0<\e(\la)\le1$ decreases monotonically to zero as $\la \to +\infty$ and that $\e(2\la)\ge \frac12 \e(\la)$, $\la\ge1$. Then one has the uniform spectral
projection estimates
\begin{equation}\label{e1}
\Bigl\|\sum_{|\la-\la_j|\le \e(\la)}E_jf\Bigr\|_{L^{\frac{2n}{n-2}}(M)}\le C\e(\la)\la \|f\|_{L^{\frac{2n}{n+2}}(M)}, \quad \la\ge 1,
\end{equation}
if and only if  one has the uniform resolvent estimates
in the region where $|\Im \, \zeta|\ge  (\Re  \zeta)^{\frac12} \e(\Re  \zeta)$, $\Re  \zeta\ge1$,
i.e.,
\begin{multline}\label{e2}
\|u\|_{L^{\frac{2n}{n-2}}(M)}\le C\bigl\| (\Delta_g+(\la+i\mu)^2)u\|_{L^{\frac{2n}{n+2}}(M)},
\\
\la, \mu \in \R, \, \, \la \ge 1, \, \, |\mu|\ge \e(\la), \, \, \, u\in C^\infty(M).
\end{multline}
\end{theorem}

It is not difficult to see that in  the case of $\Tn$ the estimate \eqref{e1} is valid for certain negative powers of $\lambda$.  One can use the lattice point counting argument of Hlawka~\cite{Hlawka}
(see \S 3.5 in \cite{SoggeHang}) to see that for $p=\infty$ \eqref{1.9} is valid with $\e(\la)=\la^{-\frac{n-1}{n+1}}$.  
Using these bounds and a simple interpolation argument one can see that \eqref{e1} and hence \eqref{e2} is valid with $\e(\la)=\la^{-\frac1{n+1}}$.
Using the multilinear estimates of Bennett, Carbery and Tao \cite{B-C-T} and recent techniques of the first author and Guth \cite{B-G}
as well as techniques from his recent work \cite{B1}-\cite{B2} one can do better and therefore obtain the following
%A corresponding
%estimate for $p=\frac{2n}{n-1}$ will allow us to prove the following 
power improvements over the earlier ones
of Shen~\cite{Shen}.

\begin{theorem}\label{theorem1.3}  If $\Delta_{\Tn}$ is the Laplacian on the flat torus $\Tn=\Rn/{\mathbb Z}^n$ of 
dimension $n\ge3$, then \eqref{e1} is valid with 
$$\e(\la)=\la^{-\e_n}$$
for some $\e_n>0$.  In fact for $n=3$ one may take any $\e_3$ satisfying
$$\e_3<\frac{85}{252}.$$
For higher odd dimensions one may take any $\e_n$ satisfying
$$\e_n<\frac{2(n-1)}{n(n+1)},$$
and for even dimensions $n\ge4$ one may take any $\e_n$ satisfying
$$\e_n<\frac{2(n-1)}{n^2+2n-2}.$$
Consequently, for this choice of $\e_n$
we have \eqref{1.1} with ${\mathcal R}$ equal to
%$${\mathcal R}_{\Tn}=
%\{\zeta \in \C: \, (\Im \, \zeta)^2\ge \delta (\Re \, \zeta)^{1-\e_n}, \, \, \Re \, \zeta \ge \delta\}
%\cup {\mathcal R}^-_\delta.$$ 
$${\mathcal R}_{\Tn}=
\{\zeta \in \C: \, |\Im \,\zeta|\ge \delta (\Re \, \zeta)^{\frac12-\e_n}, \, \, \Re \, \zeta \ge \delta\}
\cup {\mathcal R}^-_\delta.$$ 
\end{theorem}

We remark that even though the argument of Hlawka~\cite{Hlawka} shows that \eqref{1.9} is valid for $p=\infty$
with $\e(\la)=\la^{-\frac{n-1}{n+1}}$, the best we are able to do for $p=\frac{2n}{n-2}$ is $\e(\la)=\la^{-{\e_n}}$,
with $\e_n\to 0$ as $n\to \infty$.  It would
be interesting whether for the torus there can be further improvements.  In the case of $2$-dimensions, this may be related
to a theorem of Zygmund~\cite{Zygmund}, which says that $L^2({\mathbb T}^2)$-normalized eigenfunctions have bounded
$L^4$-norms, while in higher dimensions further improvements might be related to recent work of 
the first author \cite{B2}.

Besides the torus, it is known that for $p=\infty$ the arguments of B\'erard~\cite{Berard} show that one may take
$\e(\la)=(\log(2+\la))^{-1}$ in \eqref{1.9} (see \S 3.6 in \cite{SoggeHang}).  Correspondingly, we have the following

\begin{theorem}\label{theorem1.4}  Fix a Riemannian manifold $(M,g)$ of dimension $n\ge3$ with nonpositive
sectional curvatures.  Then given $\delta>0$ we have \eqref{1.1} with ${\mathcal R}$ equal to
%$${\mathcal R}_{neg}=
%\{\zeta\in \C: \, (\Im \, \zeta )^2\ge \delta\, \Re \, \zeta/(\log(2+\Re \, \zeta))^2, \, \, \, \Re \zeta\ge \delta\}
%\cup {\mathcal R}^-_\delta.$$
$${\mathcal R}_{neg}=
\{\zeta\in \C: \, |\Im \, \zeta|\ge \delta\, (\Re \, \zeta)^{\frac12}/(\log(2+\Re \, \zeta)), \, \, \, \Re \zeta\ge \delta\}
\cup {\mathcal R}^-_\delta.$$
\end{theorem}

Our proof requires the special case of $p=\frac{2n}{n-2}$ of a recent result of Hassell and Tacy~\cite{HT} showing that under
the assumption of negative curvature \eqref{1.9} is valid for {\em all} $p>\frac{2(n+1)}{n-1}$ if
$\e(\la)=(\log(2+\la))^{-1}$ (with constants depending on $p$).  For the sake of completeness, we shall present the proof
in \S5 where Theorem~\ref{theorem1.4} will be proved.  Related results to those of \cite{HT} in $2$-dimensions for
the complementary range of exponents $2<p<6$ 
were also recently obtained by Zelditch and the third author~\cite{SZ2}, and related results for eigenfunction restriction
bounds were also obtained by X. Chen~\cite{XC}.

In the  figure below we draw the  boundary of the regions described by the above results.

\begin{figure}[placement h]
\begin{center}
\includegraphics[scale=0.30]{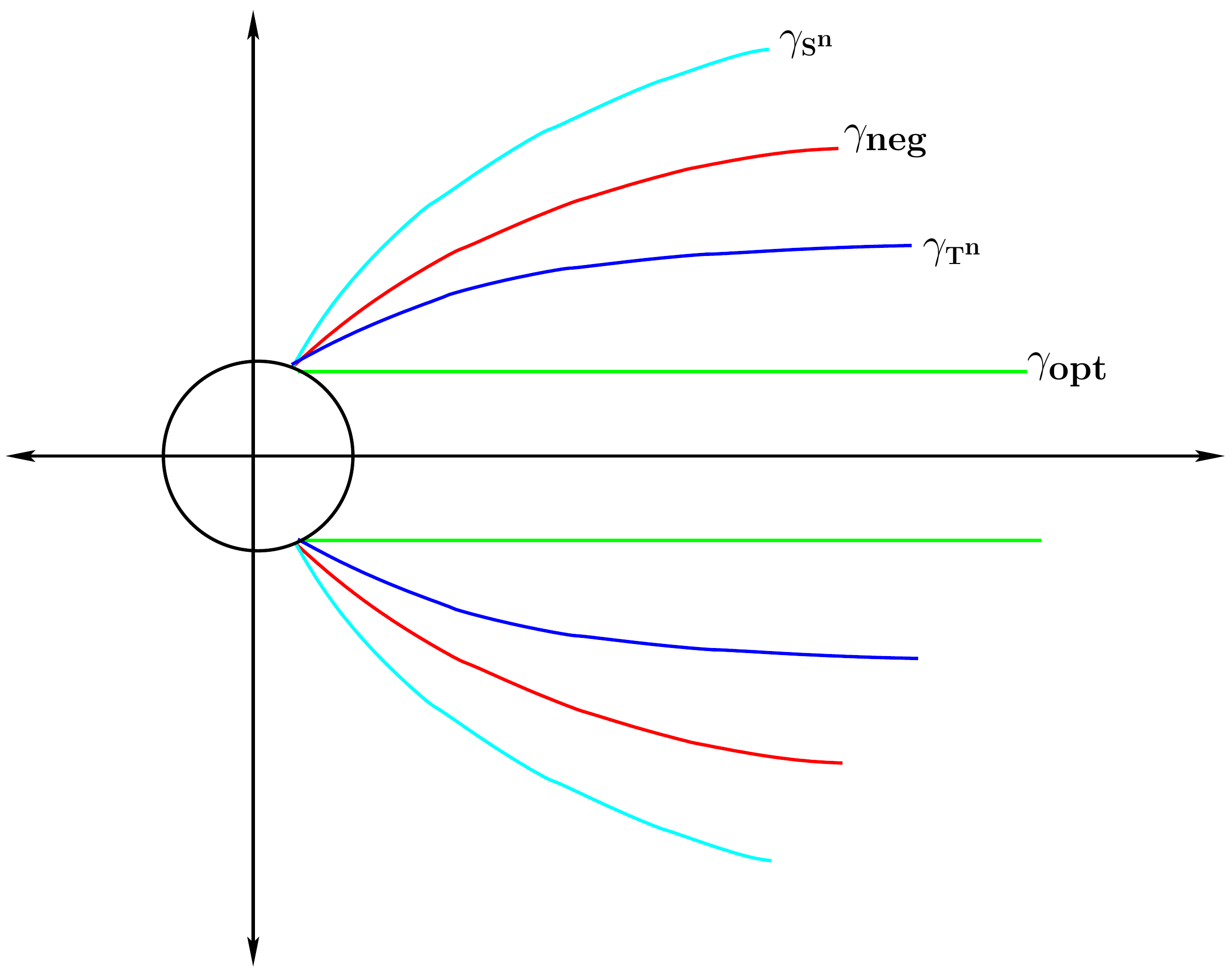}
\end{center}
\caption{Current results for $(\Delta_g+\zeta)^{-1}$ }
\end{figure}

Our techniques are similar to the earlier ones in \cite{Sogge1}, \cite{KRS}, \cite{Kenig}, \cite{Shen} and \cite{ShenZhao} in many ways,
but there are some important differences.  First, instead of using a Hadamard parametrix for $(\Delta_g+\zeta)^{-1}$, we use
short-time ones for $\cos tP$, $P=\sqrt{-\Delta_g}$, defined for $f\in C^\infty(M)$ by
\begin{equation}\label{1.13}
\bigl(\cos tP\bigr)f=\sum_{j=0}^\infty \cos t\la_j E_jf.
\end{equation}
Doing so allows us to adopt arguments of the third author~\cite{Sogge2} that were originally used to prove certain sharp results for Riesz means
on compact manifolds.  The approach is also simplified by writing $\zeta=z^2$ so that we can write
$(\Delta_g+z^2)^{-1}$ in a simple formula that involves $\cos tP$ and the extension of the Fourier transform of the Poisson
integral kernel to $\C \backslash \R$.  These modifications allow us to obtain our main estimates just by using 
stationary phase and not appealing to oscillatory integral theorems, which is useful in proving the improved
results in Theorems~\ref{theorem1.3} and \ref{theorem1.4}.

When we write $\zeta=z^2$, we shall take $z=\la+i\mu$ where $\la\in \R$ and (usually)
$\mu\in \R\backslash \{0\}$.  As we shall review in the beginning of \S2, it is easy to prove the resolvent estimates when $|\Re z|$ is bounded
above by a fixed constant.  Thus, we shall concern ourselves with the situation where it is
large and since we are squaring $z$, it suffices to treat the case where $\Re z \gg 1$.  Then the
question is how small can $|\mu|$ be so that we can obtain uniform estimates.  In Theorem~\ref{theorem1.1}, $\mu$ then essentially becomes $\e(\la)$, and so, as stated in Theorem~\ref{theorem1.2} we cannot take $\mu(\la)=o(1)$ as $\la\to \infty$ for the sphere or Zoll manifolds.
For the positive results, Theorem~\ref{theorem1.3} says that we can take $\mu(\la)=
\la^{-\e_n}$ for $\Tn$, and Theorem~\ref{theorem1.4} says that under the assumption
of nonpositive curvature we can take it to be $1/\log \la$.  Finally, the curve $\gamma_{opt}$
which corresponds the the region in the problem raised by Dos Santos Ferreira, Kenig and Salo~\cite{Kenig} corresponds to $\mu(\la)=1/\la$ in this notation, and we already indicated why
Theorem~\ref{theorem1.2} indicates how difficult it would be to obtain results  close to this
one.  In the following figure, we show the regions involving uniform estimates for 
$(\Delta_g+(\la+i\mu)^2)$.

This paper is organized as follows.  In the next section, we prove estimates for a localized version of
$(\Delta_g+z^2)^{-1}$ that are {\em uniform} in $z\in \C\backslash \R$.  This reduces bounds for $(\Delta_g+z^2)^{-1}$ to
proving remainder estimates for the difference of the actual inverse and the localized version.  These remainder terms
are handled by \eqref{1.9}.  After proving the localized estimates, we give the simple proof of
Theorem~\ref{theoremnecsuff}.
We then successively treat the results for the sphere $S^n$, the torus $\Tn$ and manifolds with
nonpositive curvature.  Also, in what follows, $C$ denotes a finite positive constant which might change at each
occurrence, although we shall be careful with issues of uniformity.

\begin{figure}[placement h]
\begin{center}
\includegraphics[scale=0.50]{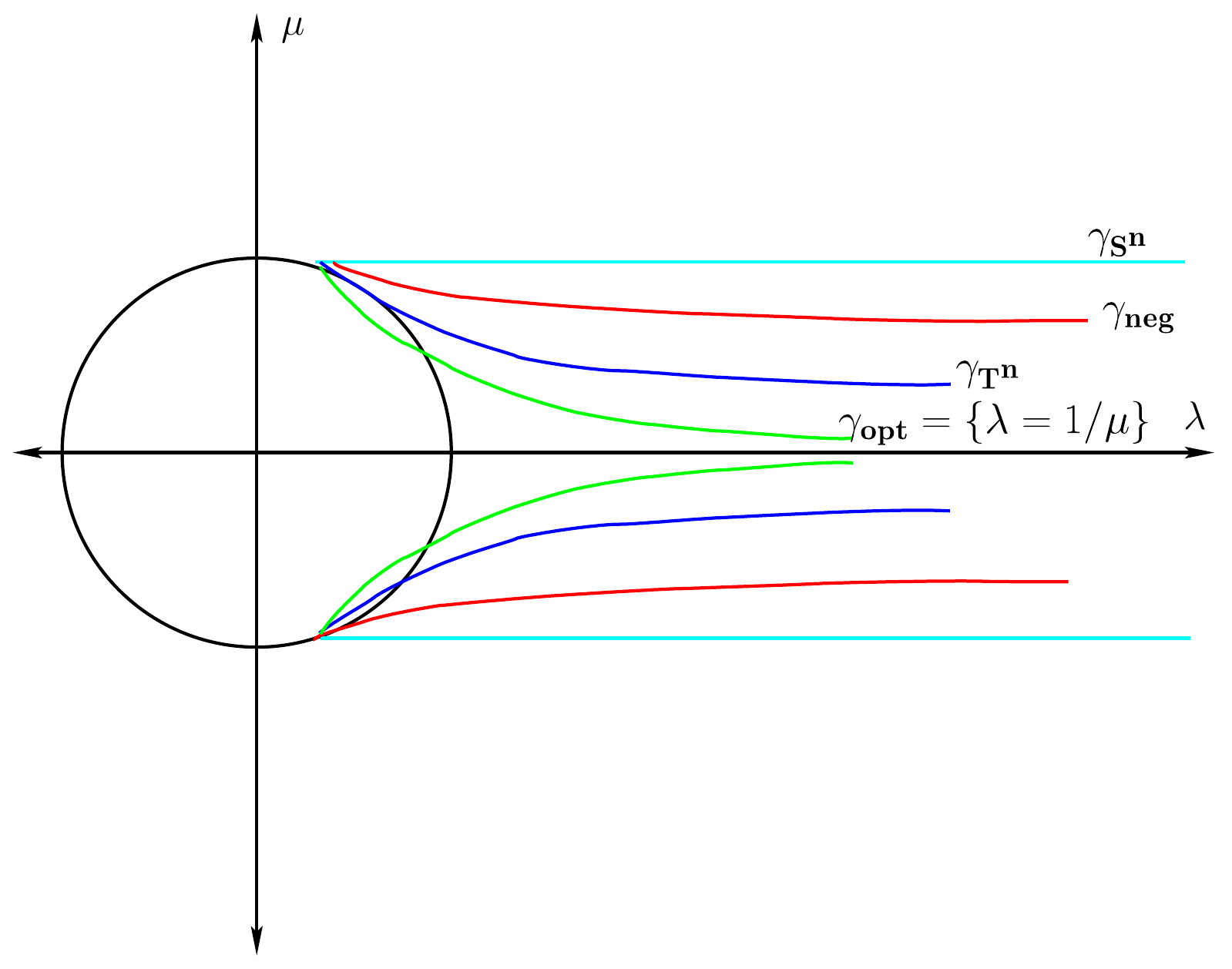}
\end{center}
\caption{Boundary of regions for $((\Delta_g+(\la+i\mu)^2)^{-1}$}
\end{figure}

%\begin{figure}
%\begin{center}
%\includegraphics[scale=0.5]{Fig02.pdf}
%\end{center}
%\caption{${\mathcal R}(\Delta_g+z^2)$}
%\end{figure}

\newsection{Uniform bounds for a local term}

The purpose of this section is to prove {\em uniform} estimates for a localized version
of the resolvents $(\Delta_g+\zeta)^{-1}$, $\zeta \in \Csl$.  These are the natural 
analogs for the current setting of the uniform Sobolev estimates for the Euclidean
case of Kenig, Ruiz and the third author~\cite{KRS}.

To state these let us derive a natural formula for the Sommerfeld-Green's function
(cf. \cite{A}) for the resolvent
$$\mathfrak{S}(x,y)=\sum_j\frac{e_j(x)e_j(y)}{\zeta-\lambda^2_j}, \quad
x,y\in M, \, \, \, \zeta \in \Csl.$$
Since the formula involves $\la_j^2$, and since the map $\C \ni z \to z^2\in \C$ is onto,
it is natural to write $\zeta$ as a square
$$\zeta=(\la+i\mu)^2, \quad \la,\mu\in \R, \, \, \mu\ne 0$$
so that we can use the following

\begin{lemma}\label{Ft}
For all $\mu\in \R$, $\mu\ne 0$
\begin{equation}\label{a}
\int_{-\infty}^\infty \frac{e^{-i\tau t}}{\mu-i \tau} \, d\tau = 2\pi \, \sgn \mu H(\mu t)e^{-\mu t},
\end{equation}
where $H(t)$ is the Heaviside function which equals one for $t\ge 0$ and $0$ for $t<0$.
Also, for all $\lambda, \mu\in \R$, $\mu\ne 0$,
\begin{equation}\label{b}
\int_{-\infty}^\infty
\frac{e^{-i\tau t}}{\tau^2-(\lambda +i\mu)^2} \, d\tau
=\frac{i\pi \, \mathrm{sgn } \mu}{\lambda +i\mu}e^{i(\mathrm{sgn } \mu)\lambda |t|}e^{-|\mu t|}.
\end{equation}
\end{lemma}

\begin{proof}  To prove \eqref{a}, we may assume that $\mu>0$ since the other case follows from reflection.  Then if $0<\varepsilon<\mu$,
Cauchy's integral formula yields
$$\int_{-\infty}^\infty \frac{e^{-i\tau t}}{\mu-i \tau} \, d\tau = e^{-(\mu-\varepsilon)t}
\int_{-\infty}^\infty \frac{e^{-it\tau}}{\varepsilon -i\tau} \, d\tau.$$
Letting $\e \searrow 0$, the right side tends to $e^{-\mu t}$ times the Fourier transform of $i(\tau+i0)^{-1}$, and since the 
latter is $2\pi H(t)$, we get  \eqref{a}.

Since
$$\frac1{\tau^2-(\lambda+i\mu)^2}=\frac1{2(\lambda+i\mu)}
\Bigl(\, \frac1{\tau-\lambda-i\mu}-\frac1{\tau+\lambda+i\mu}\, \Bigr),$$
\eqref{a} implies \eqref{b}.  Alternatively, since for $\mu>0$ the right side of \eqref{b} is
$$\frac{\pi i}{\lambda+i\mu}e^{i(\lambda+i\mu)|t|},$$
we conclude that for $\lambda=0$ and $\mu>0$, \eqref{b} is just the formula for the Fourier transform of the Poisson kernel:
$$\frac1\pi\int_{-\infty}^\infty e^{-i\tau t}\frac\varepsilon{\tau^2+\varepsilon^2}\, d\tau=e^{-\varepsilon |t|},
\quad
\varepsilon>0.$$
Since both sides of \eqref{b} are analytic in the half-plane $\{z=\lambda+i\mu, \, \mu>0\}$, we conclude that \eqref{b} must be valid
when $\mu>0$.  The formula for $\mu<0$ follows from this and the fact that the left side of \eqref{b} is an even function
of $z=\lambda+i\mu$.
\end{proof}

Using the formula \eqref{b} we get from Fourier's inversion formula
\begin{multline*}
\bigl(\lambda_j^2 -(\lambda+i\mu)^2\bigr)^{-1}=\frac{\pi i \, \mathrm{sgn } \mu}{2\pi (\lambda+i\mu)}
\int_{-\infty}^\infty e^{i(\text{sgn}\, \mu)\lambda |t|} e^{-|\mu t|}\, e^{i\lambda_j t}\, dt
\\
=\frac{i\, \mathrm{sgn } \mu}{\lambda+i\mu}\int_0^\infty  e^{i(\text{sgn}\, \mu)\lambda t}e^{-|\mu|t}\, \cos t\lambda_j \, dt.
\end{multline*}
Consequently, 
if , as before, we set,
$P=\sqrt{-\Delta_g}$,
 then we have for $f\in C^\infty$ the formula
\begin{equation}\label{c}\bigl(\Delta_g+(\lambda+i\mu)^2\bigr)^{-1}f=
\frac{ \mathrm{sgn } \mu}{ i(\lambda+i\mu)}\int_0^\infty e^{i(\text{sgn}\, \mu) \lambda t}e^{-|\mu|t}\, (\cos tP) f \, dt.
\end{equation}
In other words, the Sommerfeld-Green kernel for the resolvent,
$(\Delta_g+(\la+i\mu)^2)^{-1}$, is given by the formula
\begin{equation}\label{somkernel}
\So(x,y)=\frac{\mathrm{sgn } \mu}{i(\la+i\mu)}\int_0^\infty e^{i(\text{sgn}\, \mu) \lambda t}e^{-|\mu|t}\,
\sum_{j=0}^\infty \cos t\la_j e_j(x)e_j(y)\, dt.
\end{equation}
Here $\cos tP$ is as in \eqref{1.13}.  Put another way, $u(t,x)=(\cos tP f)(x)$ is the solution
of the Cauchy problem for the wave operator on $(M,g)$,
\begin{equation*}
\begin{cases}(\partial_t^2-\Delta_g)u(t,x)=0
\\
u(0,\cd)=f, \quad \partial_tu(0,\cd)=0.
\end{cases}
\end{equation*}

To define the localized resolvent operators, we fix a function $\rho\in C^\infty(\R)$ satisfying
\begin{equation}\label{2.2}
\rho(t)=1, \, \, t\le \delta_0/2, \, \, \, \rho(t)=0, \, \, \, t\ge \delta_0, %\, \, \rho(t)=\rho(-t),
\end{equation}
where
$$\delta_0=\min(1, \text{Inj }M/2),$$
with $\text{Inj }M$ denoting the injectivity radius of $(M,g)$.  The localized 
versions of resolvent operators then
are given by 
\begin{equation}\label{2.3}
\Rzeta =\frac{\mathrm{sgn } \mu}{i(\la+i\mu)}\int_{0}^\infty \rho(t)\Poi \cos tP\, dt,
\end{equation}
and we have the following uniform KRS-type estimates for them:

\begin{theorem}\label{theorem2.1}  Fix a compact boundaryless Riemannian manifold $(M,g)$ 
of dimension $n\ge3$.  Then, given $\delta>0$, there is a constant $C_\delta$ so that
\begin{equation}\label{2.4}
\|\Rzeta f\|_{\Lleft}\le C_\delta \|f\|_{\Lright},
\end{equation}
for all
\begin{equation}\label{2.5}
\la, \mu\in \R, \, \, \mu\ne 0, \, \, \,  |\la+i\mu|\ge \delta.
\end{equation}
\end{theorem}

At the end of this section we shall give a simple argument showing how the estimates
of Dos Santos Ferreira, Kenig and Salo~\cite{Kenig} are an immediate corollary of 
this result and the spectral projection bounds in \cite{Sogge1}.  Later we shall also use
Theorem~\ref{theorem2.1}  to obtain the improvements of the resolvent bounds in \cite{Kenig}  and \cite{Shen}, mentioned in the introduction.

Let us first realize that it is easy to estimate the operators in \eqref{2.3} when $|\la|$ is 
bounded by a fixed constant $A$, i.e., $|\la|\le A$.  We note a simple integration
by parts argument shows  that the function
\begin{equation}\label{multiplier}
\R\ni \tau \longrightarrow
%{\mathcal R}_{\la,\mu}(\tau)
\mathfrak{S}_{\la,\mu}(\tau)
=
\frac{\mathrm{sgn } \mu}{i(\la+i\mu)}\int_0^\infty \rho(t)\Poi \cos t\tau\, dt,\end{equation}
is uniformly bounded by $C(1+\tau^2)^{-1}$, under this assumption if $|\la+i\mu|\ge\delta$, with $C=C_{\delta,A}$ depending
on $A$ and $\delta>0$.  Consequently, we can obtain \eqref{2.4} in this case by appealing
to the following easy consequence of the spectral projection estimates of the third author.

\begin{lemma}\label{lemma2.2}  Given a fixed compact Riemannian manifold of dimension
$n\ge3$ there is a constant $C$ so that whenever $\alpha \in C(\overline{\R_+})$ we have
$$\|\alpha(P)f\|_{\Lleft}\le C\Bigl(\, \sup_{\tau\in \R_+}(1+\tau^2)|\alpha(\tau)| \, \Bigr) \, 
\|f\|_{\Lright},$$
if $\alpha(P)$ is the operator defined for $f\in C^\infty(M)$ by
$$\alpha(P)f=\sum_{j=0}^\infty \alpha(\lambda_j)E_j f.$$
Also, if we let 
$$\alpha_k(P)f=\sum_{\la_j\in [k-1,k)}\alpha(\la_j)E_jf, \quad k=1,2,3,\dots,$$
then if $p\ge \frac{2(n+1)}{n-1}$ and $\sigma(p)$ is as in \eqref{1.8}
$$\|\alpha_k(P)f\|_{L^p(M)}\le Ck^{\sigma(p)}\Bigl( \, \sup_{\tau\in[k-1,k)} |\alpha(\tau)| \, \Bigr) \, \|f\|_{L^{\frac{p}{p-1}}(M)}.$$
\end{lemma}

\begin{proof}  Let us start with the second inequality.
If $\chi_\la$ are the spectral projection operators
$$\chi_\la f= \sum_{\la_j\in [\la-1, \la)}E_jf,$$
then it was shown  in \cite{Sogge1}  
that for $p\ge \frac{2(n+1)}{n-1}$ we have the following
estimates for their operator norms
$$\|\chi_\lambda\|_{L^2(M)\longrightarrow L^p(M)}=\|\chi_\lambda\|_{L^{\frac{p}{p-1}}(M)\longrightarrow L^2(M)}\le C\lambda^{\frac{\sigma(p)}2}, \quad \lambda\ge 1.$$
Therefore since $\alpha_k(P)=\chi_k\circ \alpha_k(P)\circ \chi_k$, for $k=1,2,3,\dots$, we have
\begin{align*}
\|\alpha_k(P)f\|_{L^p(M)}&\le Ck^{\frac{\sigma(p)}2}\|\alpha_k(P)\chi_k f\|_{L^2(M)}
\\
&\le Ck^{\frac{\sigma(p)}2}\bigl(\sup_{\tau\in [k-1,k]}|\alpha(\tau)| \bigr)\, \|\chi_kf\|_{L^2(M)}
\\
&\le Ck^{\sigma(p)} \bigl(\sup_{\tau\in [k-1,k]}|\alpha(\tau)| \bigr)\, \|f\|_{L^{\frac{p}{p-1}}(M)},
\end{align*}
giving us the last part of the lemma.

To prove the $\Lright \to \Lleft$ bounds for the non-truncated operator, since $\frac{2n}{n+2}<2<\frac{2n}{n-2}$, we can use
Littlewood-Paley theory to reduce to the case where the Fourier multiplier $\alpha$ has
dyadic support (see \cite{Soggebook}).  In other words, it is enough to prove the first inequality
under the additional assumption that for some $j=0,1,2,\dots$
$$\alpha(\tau)=0 \quad \text{if } \, \, \, \tau\notin [2^{j-1}, 2^j].$$
But since then
$$\|\alpha(P)f\|_{\Lleft}\le \sum_{k\in [2^{j-1},2^j]}\|\alpha_k(P)f\|_{\Lleft},$$
this desired estimate follows from what we have just done, since $\sigma(p)=1$ when $p=\frac{2n}{n-2}$.
\end{proof}

Returning to the proof of Theorem~\ref{theorem2.1}, by what we have just done,
we are left with proving the nontrivial part of \eqref{2.4}, which would be to show that
the estimate holds when $|\la|\ge 1$.  After possibly multiplying $\la+i\mu$ by $-1$,
we can simplify the notation further, and note that our task is reduced to proving
the uniform estimates under the assumption that $\la\ge 1$ and $\mu\ne0$.

To prove the bounds for this case, we shall use the Hadamard parametrix
for the wave equation and the following stationary phase estimates that
are essentially in the paper of Kenig, Ruiz and the third author~\cite{KRS}:

\begin{proposition}\label{prop2.3}  Let $n\ge2$ and assume that $a\in C^\infty(\R_+)$
satisfies the Mihlin-type condition that for each $j=0,1,2,\dots$
\begin{equation}\label{2.7}
\Bigl|\frac{d^j}{ds^j} a(s)\Bigr|\le A_j s^{-j}, \quad s>0.
\end{equation} 
Then there is a constant $C$, which depends only on the size of finitely many of the
constants $A_j$ so that for every $w\in \C \backslash \R$
\begin{equation}\label{2.8}
\left|
\int_{\Rn}\frac{a(|\xi|)e^{ix\cdot \xi}}{|\xi|-w}\, d\xi\right|
\le C\bigl(|x|^{1-n}+(|w|/|x|)^{\frac{n-1}2}\bigr).
\end{equation}
\end{proposition}

Note that this estimate immediately implies the following bounds which will be useful later
on
\begin{multline}\label{2.9}
\left| \int_{\Rn}
\frac{a(|\xi|)e^{ix\cdot \xi}}{(|\xi|+w_1)(|\xi|+w_2)}\, d\xi \right|
\\
\le \frac{C}{|w_1-w_2|}\Bigl(|x|^{1-n}+\Bigl(\frac{|w_1|+|w_2|}{|x|}\Bigr)^{\frac{n-1}2}
\Bigr), \quad
w_0,w_1\in \C\backslash \R.
\end{multline}
As in \eqref{2.8}, the constant $C$ depends only on $n$ and finitely many of the constants
in \eqref{2.7}.

\begin{proof}[Proof of Proposition~\ref{prop2.3}]
Recall that if $b\in C^\infty(\R_+)$ satisfies
$$\Bigl|\frac{d^j}{ds^j} b(s)\Bigr|\le B_j s^{-1-j}, \quad s>0,
$$
then we have that the Fourier transform of $b(|\xi|)$ is a function satisfying
$$\left|
\int_{\Rn}e^{-ix\cdot \xi} b(|\xi|)\, d\xi\right|\le B|x|^{1-n},
$$
where for each $n$, the constant $B$ depends only on the size of finitely many of the constants
$B_j$.
Therefore since the dilates of $a(|\xi|)$, $\xi\to a(\varepsilon |\xi|)$ satisfy exactly the same bounds
as in \eqref{2.7}, we conclude from a change of scale that there must be a uniform constant
$C$ so that the left side of \eqref{2.8} is $\le C|x|^{1-n}$ if $w=|w|e^{i\theta}$
and $\theta \in (-\pi,\pi)\backslash [-\pi/4,\pi/4]$.

As a result, by a change of scale, we would have \eqref{2.8} if we could show that
$$\left| \int_{\Rn}
\frac{a(|\xi|)e^{ix\cdot\xi}}{|\xi|-1+i\varepsilon}\, d\xi\right|
\le C\Bigl(|x|^{1-n}+|x|^{-\frac{n-1}2}\Bigr), \quad 0<\varepsilon <1.$$
If $\beta\in C^\infty_0(\R)$ equals one when $s\in [1/2,2]$ and is supported in
$[1/4,4]$, by the above argument, if we replace $a(|\xi|)$ by
$(1-\beta(|\xi|))a(|\xi|)$ then the resulting terms are bounded by a uniform constant
times $|x|^{1-n}$.  If we let $\alpha(|\xi|)=\beta(|\xi|)a(|\xi|)$, we conclude that it
suffices to show that if $\alpha \in C^\infty_0(\R)$ is supported in $\{s\in \R: \,
s\in (1/4,4)\}$, then
\begin{equation}\label{2.12}
\left|\int_{\Rn}\frac{\alpha(|\xi|)e^{ix\cdot \xi}}{|\xi|-1+i\varepsilon} \, d\xi\right|
\le C(1+|x|)^{-\frac{n-1}2}, \quad 0<\varepsilon <1.
\end{equation}

To prove this we note that since
$$\frac1{|\xi|-1+i\varepsilon}=\frac{|\xi|-1}{(|\xi|-1)^2+\varepsilon^2}-\frac{i\varepsilon}
{(|\xi|-1)^2+\varepsilon^2},$$
the estimate trivially holds if $|x|\le 1$.  To handle the remaining case, we recall that we
have the following stationary phase formula for the Fourier transform of the Euclidean
surface measure on $S^{n-1}$, $n\ge2$,
\begin{equation}\label{surface}\int_{S^{n-1}}e^{ix\cdot \omega}\, d\sigma(\omega)
=|x|^{-\frac{n-1}2}c_+(|x|)e^{i|x|}+|x|^{-\frac{n-1}2}c_-(|x|)e^{-i|x|},
\end{equation}
where for, say, $r\ge 1/4$, the coefficients satisfy
\begin{equation}\label{surface2}
\Bigl|\frac{d^j}{dr^j}c_+(r)\Bigr|
+\Bigl|\frac{d^j}{dr^j}c_-(r)\Bigr|
\le C_jr^{-j}, \quad j=0,1,2,\dots .
\end{equation}
Therefore, the integral in the left side of \eqref{2.12} is the sum of the two terms
$$|x|^{-\frac{n-1}2}\int_{1/4}^4 \frac{\alpha(\tau)c_{\pm}(\tau|x|)\tau^{\frac{n-1}2}}
{\tau-1+i\varepsilon} e^{\pm i \tau |x|}\, d\tau.$$
Since the Fourier transforms of 
the functions
$\tau \to \alpha(\tau)c_{\pm}(\tau|x|)\tau^{\frac{n-1}2}
\in C^\infty_0((1/4,4))
$,
have uniformly bounded $L^1$-norms when $|x|\ge1$, we get \eqref{2.12}
for such $x$ by Lemma~\ref{Ft}.
\end{proof}

Let us turn to the proof of the missing uniform bounds
\begin{equation}\label{2.1.14}
\|\Rzeta f\|_{\Lleft}\le C\|f\|_{\Lright}, \quad \la\ge1, \, \, \mu\in \R \backslash \{0\}.
\end{equation}
Note that the function $\Rzeta(\tau)$ in \eqref{multiplier} satisfies
$$|\Rzeta(\tau)|\le C\tau^{-2}, \quad \text{if } \, \, \tau\ge 2 \la \ge 1,$$
for some uniform constant $C$.  This just follows from an integration by parts
argument using \eqref{multiplier} or one can use \eqref{2.1.35} below and \eqref{multiplier}.  On account of this,
if we fix a function $b\in C^\infty(\R)$ satisfying $b(s)=1$ for $s\le 2$ and $b(s)=0$ for
$s\ge 4$, by Lemma~\ref{lemma2.2} we have the following uniform bounds
$$\bigl\| (I-b(P/\la))\circ \Rlm f\|_{\Lleft}\le C\|f\|_{\Lright}.$$
Thus, in order to prove \eqref{2.1.14}, it suffices to show that
we have the uniform bounds
\begin{equation}\label{2.16}
\|b(P/\la)\circ \Rlm f\|_{\Lleft}\le C\|f\|_{\Lright}, \quad \la\ge 1, \, \, \mu \in \R\backslash\{0\}.
\end{equation}

We shall do this by dyadically breaking up the $t$-integral in the definition of $\Rlm$, estimating the pieces in
$L^2$ and $L^\infty$ and interpolating.  To this end, let us fix a function $\beta\in C^\infty_0(\R)$ satisfying
\begin{equation}\label{2.17}
\beta(t)=0, \, \, t\notin [1/2,2], \, \, 
|\beta(t)|\le 1, \, \, \, \, 
\text{and } \, \, \sum_{j=-\infty}^\infty \beta(2^{-j}t)=1, \, \, t> 0, 
%\, \, \, \text{and } \beta(t)=\beta(-t).
\end{equation}
Then for a given $\la\ge1$ and $\mu\in \R\backslash \{0\}$, we define operators
\begin{equation}\label{2.18}
S_j f=\frac1{i(\la+i\mu)}\int_0^\infty \beta(\la 2^{-j}t)\rho(t)\Poi \cos tPf\, dt, \, \, \, j=1,2,3,\dots,
\end{equation}
and 
\begin{equation}\label{2.19}
S_0f= \frac1{i(\la+i\mu)} \int_0^\infty \tilde \rho(\la t)\rho(t) \Poi \cos tP f \, dt,
\end{equation}
with 
$$\tilde \rho(t)=\bigl(1-\sum_{j=0}^\infty \beta(2^{-j}t)\bigr)\in C^\infty(\R),$$
and consequently, $\tilde \rho(t)=0$ if $|t| \ge 4$.

By Corollary 4.3.2 in \cite{Soggebook} or multiplier theorems in \cite{SS},
$$\sup_{\la \ge1}\|b(P/\la)\|_{L^{\frac{2n}{n-2}}(M)\longrightarrow L^{\frac{2n}{n-2}}(M)}<\infty,$$
and therefore, since $\Rlm = \mathrm{sgn } \mu\sum_{j=0}^\infty S_j$, we would obtain \eqref{2.16} 
if we could show that there is a uniform constant $C$ so that for $\la\ge 1$ and $\mu\in \R\backslash \{0\}$ we have
\begin{equation}\label{2.20}
\|b(P/\la)\circ S_0f\|_{\Lleft}\le C\|f\|_{\Lright},
\end{equation}
as well as
\begin{equation}\label{2.21}
\|S_jf\|_{\Lleft}\le C2^{-\frac{j}n}\|f\|_{\Lright}, \quad j=1,2,\dots.
\end{equation}
Note that $S_j=0$ if $2^j/\la \ge 2$.

Let us start with the first estimate since it is fairly trivial.  Since the kernel of $b(P/\la)\circ \cos tP$ is
$$\sum_{\la_k\le 4\la}b(\la_j/\la)\cos t\la_j e_j(x)e_j(y),$$
and by the local Weyl law (see e.g., (4.2.2) in \cite{Soggebook}),  $\sum_{\la_k\le 4\la}|e_k(x)|^2\le C\la^n$, $x\in M$, we deduce
that
$$\bigl|\bigl(b(P/\la)\circ \cos tP\bigr)(x,y)\bigr| \le C\la^n.$$
Since $\rho(\la t)=0$ for $t\ge 4/\la$ 
%and since, by \eqref{2.11}, we have the uniform bound
%\begin{equation}\label{2.22}|\mlm(t)|\le C\la^{-1},
%\end{equation}
we deduce from \eqref{2.19} that the kernel $K_0(x,y)$ of $b(P/\la)\circ S_0$ can be uniformly bounded as follows,
$$|K_0(x,y)|\le C\la^{n-2}.$$
Since, by the finite propagation speed for the wave operator, we have $(\cos tP)(x,y)=0$ if $d_g(x,y)>t$, where
$d_g(x,y)$ is the geodesic distance between $x$ and $y$, we also have that
$$K_0(x,y)=0, \quad \text{if } \, \, d_g(x,y)\ge 4/\la.$$
These two facts about $K_0$ lead to \eqref{2.20} after an application of Young's inequality.

We are now left with proving \eqref{2.21}.  By interpolation, we would get this estimate if we could establish the uniform
bounds
\begin{equation}\label{2.23}
\|S_jf\|_{L^2(M)}\le C\la^{-1}(2^j/\la)\|f\|_{L^2(M)},
\end{equation}
and
\begin{equation}\label{2.24}
\|S_jf\|_{L^\infty(M)}\le C\la^{n-2}2^{-\frac{n-1}2 j}\|f\|_{L^1(M)},
\end{equation}
since $\tfrac{n-2}{2n}=\tfrac12 \cdot \tfrac{n-2}n$ and
$$2^{-\frac{j}n}=\bigl(\la^{-1}(2^j/\la)\bigr)^{\frac{n-2}n}(\la^{n-2}2^{-j\frac{n-1}2}\bigr)^{\frac2n}.$$

The first estimate is easy.  If we use \eqref{2.18}, the fact that $\cos tP$ is bounded on $L^2$ with norm $1$ and 
\eqref{2.18}, we get
$$\|S_j f\|_{L^2}
\le \la^{-1}\int_{2^{j-1}/\la \le t\le 2^{j+1}/\la} \, \|\cos tP f\|_{L^2} \, dt
\le 4\la^{-1}(2^j/\la)\|f\|_{L^2(M)},$$
as asserted.

This leaves us with proving \eqref{2.24}, which is the same as showing the kernel $K_j$ of $S_j$ satisfies
\begin{equation}\label{2.25}
|K_j(x,y)|\le C\la^{n-2}2^{-\frac{n-1}2 j},
\end{equation}
if, by \eqref{2.18},
\begin{equation}\label{2.2.6}
K_j(x,y)=\frac1{i(\la+i\mu)}\int_0^\infty \beta(\la 2^{-j}t)\rho(t)\Poi \bigl(\cos tP\bigr)(x,y)\, dt.
\end{equation}

To proceed, we need to use the Hadamard parametrix.  Since $\rho(t)=0$ when $t$ is larger than half the injectivity
radius, the Hadamard parametrix says, that, modulo a smooth error, on the support of $\rho$, we have
\begin{equation}\label{Had}\bigl(\cos tP\bigr)(x,y)=\sum_\pm \IRn e^{i\kappa(x,y)\cdot \xi}
e^{\pm it|\xi|} \alpha_{\pm}(t,x,y,|\xi|)\, d\xi,
\end{equation}
where $\kappa(x,y)$ denotes local geodesic coordinates of $x$ about $y$ so that
$$|\kappa(x,y)|=d_g(x,y),$$
and the symbol satisfies
\begin{equation}\label{2.27}
|D^{\gamma_1}_\xi D^{\gamma_2}_{t,x,y} \alpha_\pm(t,x,y,|\xi|)|\le C_{\gamma_1, \gamma_2}(1+|\xi|)^{-|\gamma_1|},
\end{equation}
for all multi-indices $\gamma_j$, $j=1,2$.  See \cite{Had}, \cite{Hor3}, \cite{RieszHad} or \cite{SoggeHang}.
If we replace $(\cos tP)(x,y)$ by the smooth error in the Hadamard parametrix, the resulting expression will be
$O(\la^{-1}(2^j/\la))$ by \eqref{2.2.6}, which is better than the desired bounds for $K_j$, assuming as we are
that $j\ge1$ and $2^j/\la \le 2$.

Thus, if we set
$$\e = 2^j/\la, \quad  \la^{-1}\le \e \le 2,$$
it suffices to show that the kernels $K^\pm_j(x,y)$ given by
\begin{multline*}
\frac1{i(\la+i\mu)}\int_0^\infty \IRn \beta(t/\e) \rho(t) \Poi\alpha_\pm(t,x,y,|\xi|) e^{i\kappa(x,y)\cdot \xi \pm it|\xi|}\, d\xi dt
\\
%=\frac{\e^{1-n}}{2\pi}\int_{-\infty}^\infty 
=\frac{\e^{1-n}}{i(\la+i\mu)}\int_0^\infty\IRn \beta(t)\rho(\e t) \Poie
\alpha_\pm (\e t, x,y, |\xi|/\e) e^{i\frac{\kappa(x,y)}{\e} \cdot \xi \pm it|\xi|} \, d\xi dt
\end{multline*}
satisfy
\begin{equation}\label{2.28}
|K^\pm_j(x,y)|\le C\la^{\frac{n-3}2}\e^{-\frac{n-1}2}, \quad \text{if } \, \, \la^{-1}\le \e \le 2.
\end{equation}

If $t\approx 1$ then using \eqref{2.27} and a simple integration by parts argument we see that we have the uniform
bounds
$$\Bigl|\IRn e^{iv\cdot \xi \pm it|\xi|} \alpha_\pm(\e t,x,y,|\xi|/\e)\, d\xi\Bigr| \le C,
\quad \text{if } \, \, \, |v|/t\notin [1/2,2].$$
From this and the support properties of $\beta$ we deduce that
$$|K^\pm_j(x,y)|\le C\e^{1-n}\la^{-1}, \quad \text{ if} \, \, \, 
|v_\e|=|\kappa(x,y)|/\e \notin [1/4,4],$$
which is better than the bounds in \eqref{2.28} since we are assuming that $\e\ge \la^{-1}$.

Consequently, we have reduced matters to proving \eqref{2.28} under the assumption that 
$1/4 \le |v_\e|\le 4$.  To this, we note that if $a^\pm_\e(\tau,x,y,|\xi|)$  is the inverse Fourier transform of
$$t\to \beta(t)\rho(\e t) \alpha_\pm(\e t,x,y,|\xi|/\e),$$
which, by \eqref{2.27} satisfies
\begin{equation}\label{2.29}
|D_\xi^\gamma a_\e^\pm(\tau, x,y,\xi)|\le C_{N,\gamma}(1+|\tau|)^{-N}|\xi|^{-|\gamma|},
\end{equation}
for all $\e$ and every $N$ and $\gamma$, then, by \eqref{b},
\begin{align*}
\mathrm{sgn } \mu\,  K^\pm_j(x,y)&=\e^{1-n}\int_{-\infty}^\infty \left( \, 
\IRn e^{iv_\e \cdot \xi} \frac{\e^{-1} a^\pm_\e(\tau,x,y,|\xi|)}
{(\pm |\xi|-\tau)/\e)^2-(\la+i\mu)^2} \, d\xi \, \right) \, d\tau
\\
&=\e^{2-n}\int_{-\infty}^\infty \left(\, \IRn e^{iv_\e \cdot \xi} 
\frac{a^\pm_\e(\tau,x,y,|\xi|)}
{(\pm |\xi|-\tau -\e \la -i\e \mu)(\pm |\xi|-\tau +\e \la +i\e \mu)}\, d\xi \, \right)
\, d\tau.
\end{align*}
Since we are assuming now that $|v_\e|=|\kappa(x,y)/\e|\approx 1$, it follows from
\eqref{2.9} and \eqref{2.29} that for each $N\in {\mathbb N}$ there must be a constant
$C_N$ so that
\begin{multline*}
\left| \, \IRn e^{iv_\e \cdot \xi}
\frac{a^\pm_\e(\tau,x,y,|\xi|)}
{(\pm |\xi|-\tau -\e \la -i\e \mu)(\pm |\xi|-\tau +\e \la +i\e \mu)}\, d\xi \, \right|
\\
\le C_N(1+|\tau|)^{-N} |\e \la|^{-1}\, (1+ |\tau|+ |\e \la|)^{\frac{n-1}2},
\end{multline*}
Thus, if we choose $N>\frac{n-2}2+2$, then by the previous inequality and our assumption that $\e\la\ge1$
$$|K^\pm_j(x,y)|\le C\e^{2-n} (\e \la)^{-1}\, (\e \la)^{\frac{n-1}2} = C\la^{\frac{n-3}2} \e^{-\frac{n-1}2},
$$
which is \eqref{2.28} for this remaining case.

This completes the proof of Theorem~\ref{theorem2.1}. \qed

\bigskip

\subheading{Global resolvent estimates at the unit scale}

Let us now see how  our localized estimates imply the following unit scale estimates.

\begin{theorem}\label{theorem2.5}  Fix a compact boundaryless Riemannian manifold of dimension $n\ge3$.  Then given
$\delta>0$ there is a constant $C$ so that if
\begin{equation}\label{2.30} z=\lambda+i\mu \, \, \, \text{with } \, \, \la, \mu \in {\mathbb R}, \, \, 
\la, |\mu|\ge \delta,
\end{equation}
then
\begin{equation}\label{2.31}
\|u\|_{\Lleft}\le C\|(\Delta_g +z^2)u\|_{\Lright}, \quad u\in C^\infty(M).
\end{equation}
Additionally, if we let
$$\chi_{[\la-\delta,\la+\delta]}f=\sum_{\la_j\in [\la-\delta,\la+\delta]}E_j f,$$
then for fixed $0<\delta \le 1$ we have the uniform bounds
\begin{multline}\label{2.32}
\bigl\| (I-\chi_{[\la-\delta,\la+\delta]})\circ (\Delta_g+z^2)^{-1}f\bigr\|_{\Lleft}\le C\|f\|_{\Lright},
\\
 z=\la+i\mu, \, \, \la \ge 1, \, \, \mu\in \R\backslash \{0\}.
\end{multline}
\end{theorem}

Inequality \eqref{2.31} implies the earlier unit-scale resolvent estimates of \cite{Kenig} and \cite{Shen}, since, as noted
before Proposition~\ref{prop2.3} yields uniform resolvent estimates for $\Delta+\zeta$ when
$\text{Re }\zeta\le \delta$ and $|\zeta|\ge \delta$.  Also, after scaling the metric, we see that in proving
\eqref{2.31}, we may assume that $\delta=1$.

\begin{proof}[Proof of Theorem~\ref{theorem2.5}]
Let us start by proving \eqref{2.31}.  As we just noted, we may assume that
$$z=\la+i\mu, \quad \la, \mu\in \R, \, \, \la, |\mu|\ge 1.$$
We then wish to show that there is a uniform constant $C$ so that
\begin{equation}\label{2.33}
\|(\Delta_g+z^2)^{-1}f\|_{\Lleft}\le C\|f\|_{\Lright}.
\end{equation}
By \eqref{2.4}, if we let
\begin{equation}\label{2.34}
r_{\la,\mu}(P)=(\Delta_g+z^2)^{-1}-\Rzeta, \quad z=\la + i\mu,
\end{equation}
then we would obtain this estimate if 
 we could establish the following uniform bounds for this remainder term
\begin{equation}\label{2.35}
\|r_{\la,\mu}(P)f\|_{\Lleft}
\le C\|f\|_{\Lright}, \quad \la, \mu \in \R, \, \, \la, |\mu|\ge 1.
\end{equation}

To prove this, we use \eqref{2.3} to see that the multiplier that is associated
to $r_{\la,\mu}(P)$ is just the function
\begin{equation}\label{2.1.34}
r_{\la,\mu}(\tau)=\frac{\mathrm{sgn } \mu}{i(\la+i\mu)}\int_0^\infty (1-\rho(t)) \Poi \, \cos t\tau \, dt.
\end{equation}
Consequently, since $(1-\rho(t))$ vanishes near the origin, if we use Euler's formula
to write $\cos t\tau =\frac12(e^{it\tau}+e^{-it\tau})$ and apply a simple integration
by parts argument, we deduce that for every $N=1,2,3,\dots$ there is a uniform
constant $C_N$ so that
\begin{multline}\label{2.1.35}
|r_{\la,\mu}(\tau)|\le C_N\la^{-1}
\Bigl((1+|\la-\tau|)^{-N}+(1+|\la+\tau|)^{-N}\Bigr), 
\\
\tau,\la,\mu\in \R, \quad \la, \, |\mu|\ge 1.
\end{multline}
From this we deduce \eqref{2.35}, since by the second part of Lemma~\ref{lemma2.2}
$$\|r_{\la,\mu}(P)\|_{L^{\frac{2n}{n+2}}\to L^{\frac{2n}{n-2}}}
\le C\sum_{k=1}^\infty k\la^{-1}(1+|\la-k|)^{-3}\le C'.$$

To complete the proof of Theorem~\ref{theorem2.5} we need to prove \eqref{2.32}.    We could
do so by adapting the proof of the unit-scale estimates \eqref{2.31}.  However, it is very straightforward
to do so just by combining \eqref{2.31} with the spectral projection estimates of the third author \cite{Sogge1}.

Note that, by Lemma~\ref{lemma2.2},
$$\bigl\|  \chi_{[\la-\delta,\la+\delta]} \circ (\Delta_g+(\la +i\mu)^2)^{-1}\bigr\|_{L^{\frac{2n}{n+2}}\to L^{\frac{2n}{n-2}}} \le C,
\quad \text{if } \, \la \ge 1 \, \, \text{and } \, |\mu|\ge1,$$
which, along with \eqref{2.31}, implies \eqref{2.32} for $\la, |\mu|\ge1$.
Since the second part of Lemma~\ref{lemma2.2} also yields the uniform bounds
\begin{multline*}\Bigl\| (I-\chi_{[\la-\delta,\la+\delta]})\circ \bigl((\Delta_g+(\la+i)^2)^{-1}-(\Delta_g+(\la+i\mu)^2)^{-1}\bigr)\Bigr\|_{L^{\frac{2n}{n+2}}\to L^{\frac{2n}{n-2}}}
\le C_\delta, 
\\ 
\la\ge 1, \, \, \, \mu \in [-1,1]\backslash \{0\},
\end{multline*}
we therefore get \eqref{2.32}, for the remaining range where $\mu\in (-1,1)\backslash \{0\}$.
\end{proof}

\subheading{Improved global estimates from small-scale spectral projection estimates: Proof of Theorem~\ref{theoremnecsuff}   }

We shall first see how \eqref{e1} implies \eqref{e2}.  In view of Theorem~\ref{theorem2.5}, to prove \eqref{e2}, 
it suffices to just consider the case where
\begin{equation}\label{d3}
\e(\la)\le \mu \le 1.
\end{equation}
Also, in view of Theorem~\ref{theorem2.1}, if $\rho\in C^\infty(\R)$ satisfies
\begin{equation}\label{d4}
\rho(t)=1, \, \, t\le \delta_0/2, \quad \rho(t)=0, \, \, t\ge \delta_0,
\end{equation}
then it suffices to verify that
\begin{multline}\label{d5}
\Bigl\| \int_0^\infty e^{i\la t}\bigl(\cos t\sqrt{-\Delta_g}f\bigr) \, (1-\rho(t))e^{-\mu t} \, dt \Bigr\|_{L^{\frac{2n}{n-2}}(M)}
\\
\le C\la \|f\|_{L^{\frac{2n}{n+2}}(M)}, \quad \e(\la)\le \mu\le 1.
\end{multline}

To use \eqref{e1} to prove the resolvent estimates, we need the following simple lemma.

\begin{lemma}\label{lemmad2}  Suppose that $0<\mu\le 1$ and that $\rho$ is as in \eqref{d4}.  Then for every $N=1,2,3,\dots$ there
is a constant $C_N$ so that
\begin{multline}\label{d6}
\left|\int_0^\infty e^{i\la t\pm i\tau t}(1-\rho(t))e^{-\mu t}\, dt \right|
\\
\le C_N\left[ \, (1+|\la\pm \tau|)^{-N}+\mu^{-1}(1+\mu^{-1}|\la \pm \tau|)^{-N} \, \right].
\end{multline}
\end{lemma}

\begin{proof}  If $|\la\pm\tau|\le \mu$, the result is trivial.  So we may assume $|\la\pm \tau|\ge \mu$.  If we then integrate by parts and
use Leibnitz's rule, we find that the left side of \eqref{d6} is majorized by
$$|\la\pm \tau|^{-N}\sum_{j+k=N}
\int_0^\infty \mu^j e^{-\mu t} \,  \Bigl| \frac{d^k}{dt^k}(1-\rho(t))\, \Bigr| \, dt.$$
If $k\ne 0$, the summand is dominated by the first term in the right side of \eqref{d6}, in view of \eqref{d4}.  For the remaining case where
$j=N$ and $k=0$, it is clearly dominated by the second term in the right side of \eqref{d6}.
\end{proof}

\begin{proof}[Proof of Theorem~\ref{theoremnecsuff}]  To show that
\eqref{e1} implies \eqref{e2} we need to verify \eqref{d5}.  By Lemma~\ref{lemmad2}, the operator in the left side of
this inequality is of the form
$$\sum_{j=0}^\infty m_{\la,\mu}(\la_j)E_jf,$$
where for every $N=1,2,3,\dots$ there is a constant $C_N$ so that for $0<\mu\le 1$
\begin{equation}\label{mult}
|m_{\la,\mu}(\la_j)|\le C_N\Bigl[(1+|\la-\la_j|)^{-N}+\mu^{-1}(1+\mu^{-1}|\la-\la_j|)^{-N} \Bigr]
\end{equation}

Since our assumption \eqref{e1} implies that there is a uniform constant $C$ so that
\begin{equation}\label{mu}\Bigl\|\sum_{|\la-\la_j|\le \mu}E_jf\Bigr\|_{L^{\frac{2n}{n-2}}(M)}\le C\la \mu \|f\|_{L^{\frac{2n}{n+2}}(M)}, 
\quad \la \ge 1, \, \, \e(\la)\le \mu\le 1,\end{equation}
\eqref{d5} follows from the proof of Lemma~\ref{lemma2.2}.  

Indeed, by this proof, we first notice that if we use 
\eqref{mu} with $\mu=1$ along with \eqref{mult},  we conclude that outside
of a unit interval about $\lambda$, we have the following bounds
\begin{align*}
\Bigl\| \sum_{|\la_j-\la|\ge 1} m_{\la,\mu}(\la_j) E_jf\Bigr\|_{L^{\frac{2n}{n-2}}}&\le
\sum_{k=1}^\infty \,  \Bigl\|\sum_{\{\la_j: \, \la_j\in [k-1,k) \, \backslash \,  (\la-1,\la+1)\}}m_{\la,\mu}(\la_j)E_jf\Bigr\|_{L^{\frac{2n}{n-2}}}
\\
&\le C_N \sum_{k=1}^\infty k(1+|\la-k|)^{-N} \|f\|_{L^{\frac{2n}{n+2}}}\le C\la \|f\|_{L^{\frac{2n}{n+2}}},
\end{align*}
as desired, assuming that $N>2$.  To finish the proof of \eqref{d5}, we need to show that we also have
\begin{equation}\label{missing}
\Bigl\| \sum_{\la_j\in (\la-1,\la+1)} m_{\la,\mu}(\la_j)E_jf\Bigr\|_{L^{\frac{2n}{n-2}}}\le C\la \|f\|_{L^{\frac{2n}{n+2}}}.
\end{equation}
To do this, we note that if we use \eqref{mu} and argue as in the proof of Lemma~\ref{lemma2.2}, we find that
for $j\in {\mathbb Z}$ with  we have
\begin{multline*}\Bigl\|\sum_{\la_j\in [\la +\mu j, \la +\mu (j+1))\cap (\la-1,\la+1)} m_{\la,\mu}(\la_j) E_j f\|_{L^{\frac{2n}{n-2}}}
\\
\le \la \mu \bigl(\sup_{\tau \in [\la +\mu j, \la +\mu (j+1))} |m_{\la,\mu}(\tau)|\bigr) \, \|f\|_{L^{\frac{2n}{n+2}}}
\le \la  C_N\bigl(\mu +(1+|j|)^{-N}\bigr) \|f\|_{L^{\frac{2n}{n+2}}}.
\end{multline*}
Since 
$$\Bigl\| \sum_{|\la-\la_j| < 1} m_{\la,\mu}(\la_j)E_jf\Bigr\|_{L^{\frac{2n}{n-2}}}
\le \sum_{|j|\le 2\mu^{-1}}\Bigl\|\sum_{\la_j\in [\la +\mu j, \la +\mu (j+1))\cap (\la-1,\la+1)} m_{\la,\mu}(\la_j) E_j f\|_{L^{\frac{2n}{n-2}}},
$$
we clearly get \eqref{missing} from this.

To prove the converse, we note that if \eqref{e2} were valid, then we would have the uniform
bounds
\begin{equation}\label{d7}
\e(\la) \la \Bigl\|\sum_{j=0}^\infty \bigl((\la_j^2-\la^2+\e(\la)^2)^2+(2\e(\la)\la)^2\bigr)^{-1}E_jf\Bigr\|_{L^{\frac{2n}{n-2}}(M)}
\le C\|f\|_{L^{\frac{2n}{n+2}}(M)},
\end{equation}
due to the fact that
$$\frac{4i\e(\la)\la}{(\la_j^2-\la^2+\e(\la)^2)^2+(2\e(\la)\la)^2}=
\frac1{\la_j^2-(\la + i\e(\la))^2}-\frac1{\la_j^2-(\la - i\e(\la))^2}.$$
This and a $T^*T$ argument in turn implies that
\begin{equation*}
\sqrt{\e(\la)\la} \, \Bigl\|\sum_{j=0}^\infty 
\bigl((\la_j^2-\la^2+\e(\la)^2)^2+(2\e(\la)\la)^2\bigr)^{-\frac12}E_jf\Bigr\|_{L^2(M)}
\le C\|f\|_{L^{\frac{2n}{n+2}}(M)}.
\end{equation*}
As
$$\sqrt{\e(\la)\la} \, \bigl((\la_j^2-\la^2+\e(\la)^2)^2+(2\e(\la)\la)^2\bigr)^{-\frac12}
\ge \frac1{10} \bigl(\e(\la)\la\bigr)^{-\frac12}, \quad
\text{if } \, \, |\la-\la_j|\le \e(\la),$$
orthogonality and the preceding inequality imply that 
\begin{equation}\label{d8}
\Bigl\| \sum_{|\la_j-\la|\le \e(\la)}E_jf\Bigr\|_{L^2(M)}
\le C\sqrt{\e(\la)\la} \, \|f\|_{L^{\frac{2n}{n+2}}(M)}.
\end{equation}
Since by another $T^*T$ argument \eqref{d8} is equivalent to \eqref{e1},
the proof is complete.
\end{proof}
\

\newsection{Saturation of certain resolvent norms}

In this section we shall prove Theorems \ref{theorem1.1} and \ref{theorem1.2}.  For each it will be convenient to use the
following simple lemma.

\begin{lemma}\label{lemma3.1}  Suppose $\beta\in C^\infty_0(\R)$ satisfies
$$\beta(\tau)=0, \quad \tau\notin [1/4,4].$$
Then if $1\le q\le r\le \infty$, there is a constant $C=C(r,q)$ so that
\begin{equation}\label{3.1}
\|\beta(P/\la)f\|_{L^r(M)}\le C\la^{n(\frac1q-\frac1r)} \|f\|_{L^q(M)}, \quad \la \ge 1.
\end{equation}
\end{lemma}

We remark that by using the lemma we can verify our assertion that if \eqref{1.9} is valid for some finite exponent then it must
be valid for $p=\infty$.   We just choose a $\beta$ as in the lemma satisfying $\beta(\tau)=1$ for $\tau\in [1/2,2]$.
It then follows that for large $\lambda$ we have $\chi_{[\la-\e(\la),\la+\e(\la)]}=
\beta(P/\la)\circ \chi_{[\la-\e(\la),\la+\e(\la)]} \circ \beta(P/\la)$ and so applying the lemma twice
with exponents $(q,r)$ being equal to $(p,\infty)$ and $(1,p/(p-1))$, where $p$ as is in \eqref{1.9} yields
\begin{equation*}
\|\chi_{[\la-\e(\la),\la+\e(\la)]}\|_{L^1(M)\to L^\infty(M)}
\le C\la^{\frac{2n}p}
\|\chi_{[\la-\e(\la),\la+\e(\la)]}\|_{L^{\frac{p}{p-1}}(M)\to L^p(M)}.
\end{equation*}
Since if $\sigma(p)$ is as in \eqref{1.8}, we have $\sigma(p)+2n/p=(n-1)=\sigma(\infty)$, we conclude that \eqref{1.9} implies that
$$\|\chi_{[\la-\e(\la),\la+\e(\la)]}\|_{L^{1}(M)\to L^\infty(M)}\le C'\e(\la)\la^{n-1},$$
as we asserted before.

\begin{proof}[Proof of Lemma~\ref{lemma3.1}]  Let $K_\la(x,y)$ denote the kernel of $\beta(P/\la)$.  Then the proof of Theorem 4.3.1 in \cite{Soggebook}
shows that for every $N=1,2,3,\dots$ there is a constant $C_N$ which is independent of $\la$ so that
$$|K_\la(x,y)|\le C_N\la^n(1+\la d_g(x,y))^{-N}, \quad \la \ge 1.$$
Consequently, \eqref{3.1} follows from an application of Young's inequality.
\end{proof}

\begin{proof}[Proof of Theorem~\ref{theorem1.1}]  Since \eqref{1.11} is a special case of
Theorem~\ref{theorem1.2}, we shall only prove the second assertion in the theorem.

Let us start by handling the sphere with the standard metric.  We need to see that if $\zeta=(\la+i\mu)^2$
with $\la \gg1$ and $|\mu|\le1$ then
$$\bigl\|(\Delta_{S^n}+(\la+i\mu)^2)^{-1}\bigr\|_{L^{\frac{2n}{n+2}}(S^n)\longrightarrow \Lspace(S^n)}
\approx \bigl(\text{dist}(\la+i\mu , \Spec\sqrt{-\Delta_{S^n}})\bigr)^{-1},
$$
since the bounds for the remaining cases are a consequence of Theorem \ref{theorem2.5} and 
Lemma~\ref{lemma2.2}.

We also know from \eqref{2.32} with $\delta=1/10$ that if $\text{dist}(\la, \Spec\sqrt{-\Delta_{S^n}})\ge 1/10$ then
$(\Delta_{S^n}+(\la + i\mu)^2)^{-1}$ is bounded with norm independent of $\la$ and $\mu$.
Since the spectrum of $\sqrt{-\Delta_{S^n}}$ is $\{\sqrt{k(k+n-1)}\}$, $k=0,1,2,\dots$, the remaining
case which we need to handle is where for some large $k_0\in {\mathbb N}$ we have
$|\la-\sqrt{k_0(k_0+n-1)}|<1/10$ and then by \eqref{2.32}, we need to show that
\begin{multline}\label{d1}
\bigl\|\chi_{[\la-1/10,\la+1/10]}\circ (\Delta_{S^n}+(\la+i\mu)^2)^{-1}\bigr\|_{L^{\frac{2n}{n+2}}(S^n)\to \Lspace(S^n)}
\\
\approx \bigl|\sqrt{k_0(k_0+n-1)}-(\la+i\mu)\bigr|^{-1}, \quad
\text{if } \, |\mu|\le1, \, \, \text{and } \, \la \gg 1.
\end{multline}
Because $\sqrt{k_0(k_0+n-1)}$ is the unique eigenvalue lying in $[\la-1/10,\la+1/10]$, 
it follows that
$\chi_{[\la-1/10,\la+1/10]}$ just equals the projection operator, $H_{k_0}$ onto the
space of spherical harmonics of degree $k_0$.  Thus the operator occurring in the
left side of \eqref{d1} is just
$$\bigl(-k_0(k_0+n-1)+(\la+i\mu)^2\bigr)^{-1}H_{k_0}.$$
It was shown by the third author in \cite{Soggethesis} that
$$\|H_{k_0}\|_{L^{\frac{2n}{n+2}}(S^n)\to L^2(S^n)}\approx k_0^{1/2},$$
and since, by a $TT^*$ argument, we know from this that the $\Rspace(S^n)
\to \Lspace(S^n)$ norm of $H_{k_0}$ is comparable to $k_0$, we conclude that the
left side of \eqref{d1} is comparable to 
$$|-k_0(k_0+n-1)+(\la + i\mu)^2|^{-1}\, k_0.$$
Finally, since, by  our assumptions, for large $\la$ and for $|\mu|\le 1$ this is comparable to
$|\sqrt{k_0(k_0+n-1)}-(\la+i\mu)|^{-1}$, we have proven the second part of Theorem~\ref{theorem1.2} for the sphere.

Let us now handle the case of Zoll manifolds.  These are manifolds whose geodesics are all periodic with a common minimal
period, which we may assume to be equal to $2\pi$ after possibly multiplying the metric by a constant.  Then, by a theorem
of Weinstein~\cite{Weinstein}, there is a constant $\alpha=\alpha_M\ge0$ so that all of the nonzero eigenvalues of $-\Delta_g$ cluster around the 
values $(k+\alpha)^2$, $k=1,2,3,\dots$.  Specifically, for each $k$ there is a cluster of of dimension $d_k\approx k^{n-1}$ of
eigenvalues $\la_{k,j}^2$, $j=1,2,\dots, d_k$ for which $|(k+\alpha)-\la_{k,j}|\le A/k$, $k=1,2,3,\dots$ for a fixed constant $A$, and all
of the nonzero eigenvalues of $-\Delta_g$ are in one of the clusters.

Repeating the arguments for the sphere, we need to show that if $\la\gg 1$ and for some $k\in {\mathbb N}$ we have
$|\la-(k+\alpha)|<1/20$, then
\begin{equation}\label{3.3}
\bigl\| \chi_{[\la-1/10,\la+1/10]}\circ (\Delta_g+(\la + i\mu)^2)^{-1}\bigr\|_{\Lright \to \Lleft}
\approx |(k+\alpha)-(\la + i\mu)|^{-1},
\end{equation}
assuming that $|\mu|\le 1$ and, because of the additional assumption in the Zoll case,
\begin{equation}\label{3.4}
|\la+i\mu -(k+\alpha)|\ge C/k,
\end{equation}
with $C$ being a constant depending on the constant $A$ for the spectral gaps.

By the argument in the remark following Lemma~\ref{lemma3.1}, this would follow from showing that under the above assumptions
\begin{equation}\label{3.5}
k^{-(n-2)}\bigl\|\chi_{[\la-1/10,\la+1/10]}\circ (\Delta_g+(\la+i\mu)^2)^{-1}\bigr\|_{L^1(M)\to L^\infty(M)}
\approx |\la+i\mu-(k+\alpha)|^{-1},
\end{equation}
assuming \eqref{3.4} with $C$ sufficiently large.  Since $|\la-(k+\alpha)|<1/20$, it follows that if $\la$ is sufficiently large then
 $\la_{l,m}\notin [\la-1/0,\la+1/10]$ unless $l=k$. Therefore, assuming $\la$ is sufficiently large, it follows that if 
 $\{e_{k,j}(x)\}$, $j=1,\dots, d_k$ is an orthonormal basis of real eigenfunctions for the cluster, then the kernel of the above
 operator is
 $$\sum_{j=1}^{d_k}(-\la^2_{k,j}+(\la+i\mu)^2)^{-1}e_{k,j}(x)e_{k,j}(y).$$
 Since the $L^1(M)\to L^\infty(M)$ norm is the supremum of the kernel, we deduce that the left side of \eqref{3.5}  majorizes
 $$k^{-(n-2)}\left|\sum_{j=1}^{d_k}(-\la^2_{k,j}+(\la+i\mu)^2)^{-1}|e_{k,j}(x)|^2\right|.$$
Let $z_0=(-(k+\alpha)^2+(\la+i\mu)^2)^{-1}$ and $z_j=(-\la^2_{k,j}+(\la+i\mu)^2)^{-1}$, $j=1,2,\dots, d_k$.  Then since
$|\la_{k,j}-(k+\alpha)|\le A/k$, it follows that if the constant $C$ in \eqref{3.4} is fixed to be sufficiently large, then
$|z_0-z_j|\le \frac12 |z_0|$, $j=1,2,\dots, d_k$.  As a result, the left side of \eqref{3.5} dominates
\begin{align*}
|-&(k+\alpha)^2 +(\la +i\mu)^2|^{-1}k^{-(n-2)}\sum_{j=1}^{d_k}|e_{k,j}(x)|^2
\\
&\ge \frac{|-(k+\alpha)^2 +(\la +i\mu)^2|^{-1}k^{-(n-2)}}{\text{Vol}_g(M)}\int_M\sum_{j=1}^{d_k}|e_{k,j}(x)|^2\, dV_g
\\
&=|-(k+\alpha)^2 +(\la +i\mu)^2|^{-1}k^{-(n-2)}d_k/\text{Vol}_g(M).
\end{align*}
Since $d_k\approx k^{n-1}$ and $|-(k+\alpha)^2 +(\la +i\mu)^2|^{-1}\approx k^{-1}|(\la+i\mu)-(k+\alpha)|^{-1}$, we conclude
that the left side of \eqref{3.3} dominates the right side.  Since a similar argument using the second part
of Lemma~\ref{lemma2.2} implies the opposite
inequality, the proof is complete.
\end{proof}

\begin{proof}[Proof of Theorem~\ref{theorem1.2}]
We need to show that if there is a sequence $\tau_k\to \infty$ and $\e(\tau_k)\searrow 0$ with $\e(\tau_k)>0$ and
\begin{equation}\label{aa}
(\e(\tau_k)\tau_k^{n-1})^{-1}\big[N(\tau_k+\e(\tau_k)-N(\tau_k-\e(\tau_k)-)\bigr]\longrightarrow \infty,
\end{equation}
then
\begin{equation}\label{bb}
\bigl\|(\Delta_g+\tau_k^2+i\tau_k\e(\tau_k))^{-1}\bigr\|_{\Rnorm(M) \to \Lnorm(M)}\longrightarrow \infty.
\end{equation}
We know from \eqref{2.32} that $(I-\chi_{[\tau_k-1,\tau_k+1]})\circ (\Delta_g+\tau_k^2+i\tau_k\e(\tau_k))^{-1}$
has $\Lright \to \Lleft$ norm bounded by a uniform constant.  Consequently, we would have \eqref{bb}
if we could show that
\begin{equation*}%\label{bb}
\bigl\|\chi_{[\tau_k-1,\tau_k+1]}\circ(\Delta_g+\tau_k^2+i\tau_k\e(\tau_k))^{-1}\bigr\|_{\Rnorm(M) \to \Lnorm(M)}\longrightarrow \infty.
\end{equation*}
Arguing as in the remark following Lemma~\ref{lemma3.1} shows that the operator norm in the left is bounded from
below by $c\tau_k^{-(n-2)}$ times the $L^1(M)\to L^\infty(M)$ operator norm for some positive constant $c$.  Consequently, we would be 
done if we could show that
\begin{equation}\label{cc}
\tau_k^{-(n-2)}\bigl\|
\chi_{[\tau_k-1,\tau_k+1]}\circ(\Delta_g+\tau_k^2+i\tau_k\e(\tau_k))^{-1}\bigr\|_{L^1(M) \to L^\infty(M)}\longrightarrow \infty.
\end{equation}
The kernel of the operator is
$$\sum_{|\la_j-\tau_k|\le 1}(-\la_j^2+\tau^2_k+i\tau_k\e(\tau_k))^{-1}e_j(x)e_j(y),$$
and since the $L^1(M)\to L^\infty(M)$ norm is the supremum of the kernel, we deduce that the left side of \eqref{cc} majorizes
$$\tau_k^{-(n-2)}\sup_{x\in M}\left| \sum_{|\la_j-\tau_k|\le 1}(-\la_j^2+\tau_k^2+i\tau_k\e(\tau_k))^{-1}|e_j(x)|^2\right|.$$
Since the imaginary part of $(-\la_j^2+\tau_k^2+i\tau_k\e(\tau_k))^{-1}$ is $-\tau_k\e(\tau_k)/((\tau_k^2-\la_j^2)^2+\tau_k^2(\e(\tau_k))^2)$,
we deduce that for large $k$ we have
\begin{align*}
\tau_k^{-(n-2)}\bigl\|
|\chi_{[\la-1,\la+1]}&\circ(\Delta_g+\tau_k^2+i\tau_k\e(\tau_k))^{-1}\bigr\|_{L^1(M) \to L^\infty(M)}
\\
&\ge \tau_k^{-(n-2)}\sup_{x\in M}\sum_{|\la_j-\tau_k|\le 1}\frac{\tau_k\e(\tau_k)}{(\tau^2_k-\la_j^2)^2+\tau^2_k(\e(\tau_k))^2} \, |e_j(x)|^2
\\
&\ge \tau_k^{-(n-2)}\sup_{x\in M}\sum_{|\la_j-\tau_k|\le \e(\tau_k)}\frac{\tau_k\e(\tau_k)}{(\tau^2_k-\la_j^2)^2+\tau^2_k(\e(\tau_k))^2} \, |e_j(x)|^2
\\
&\ge \frac1{10} \tau_k^{-(n-1)}(\e(\tau_k))^{-1}\sup_{x\in M}\sum_{|\la_j-\tau_k|\le \e(\tau_k)}|e_j(x)|^2
\\
&\ge \frac1{10\text{Vol}_g(M)}\, \tau^{-(n-1)}_k(\e(\tau_k))^{-1}\int_M \sum_{|\la_j-\tau_k|\le \e(\tau_k)}|e_j(x)|^2 \, dV_g
\\
&=\frac1{10\text{Vol}_g(M)}\, \tau_k^{-(n-1)}(\e(\tau_k))^{-1}\bigl[N(\tau_k+\e(\tau_k))-N(\tau_k-\e(\tau_k)-)\bigr],
\end{align*}
and since, by assumption, the last quantity tends to $\infty$ as $k\to \infty$, we get \eqref{cc}, which completes the proof.

\end{proof}

\newsection{Improved  bounds for the torus}

We now consider the torus $\Tn = \Rn / \Zn$.  To prove Theorem~\ref{theorem1.3}, we need to show, in view of 
Theorem~\ref{theoremnecsuff}, that if $\e_n$ is as in the statement of Theorem~\ref{theorem1.3}, then
\begin{equation}\label{4.1}
\Bigl\|\sum_{|\la-\la_j|\le \e(\la)}E_jf\Bigr\|_{L^{\frac{2n}{n-2}}(\Tn)}\le C\la \e(\la)\|f\|_{L^{\frac{2n}{n+2}}(\Tn)}, \quad
\e(\la)=\la^{-\e_n}, \quad \la \ge1.
\end{equation}
Writing $\Tn \cong (-\frac12, \frac12]^n$, the eigenfunctions of $\sqrt{-\Delta_{\Tn}}$ are $e^{2\pi i k\cdot x}$, $k\in \Zn$, with
eigenvalues $2\pi |k|$.   Thus, if
\begin{equation}\label{4.02}
\Hat f(k)=\int_{(-\frac12,\frac12]^n}f(y)e^{-2\pi i k\cdot y}\, dy,
\end{equation}
an equivalent way of writing \eqref{4.1} is
\begin{multline}\label{4.03}
\Bigl\| \sum_{\{k\in \Zn: \, |\, |k|-\la|\le \e(\la)\}}\Hat f(j)e^{2\pi ik\cdot x}\Bigr\|_{L^{\frac{2n}{n-2}}(\Tn)}
\\
\le C\la \e(\la)\|f\|_{L^{\frac{2n}{n+2}}(\Tn)}, \quad \e(\la)=\la^{-\e_n}, \, \,  \la\ge1.
\end{multline}

\newcommand{\Tt}{{\mathbb T}^3}

			\numberwithin{equation}{subsubsection}
			\numberwithin{theorem}{subsubsection}

\subsubsection{Model argument for $\Tt$}\label{S4.1}

\par\noindent

As a first step, to motivate the refinements to follow, and also the arguments for the case of general manifolds of nonpositive
curvature to follow, let us prove a weaker result than the one stated in Theorem~\ref{theorem1.3} when $n=3$, by giving a simple
proof that \eqref{4.1} holds for the 3-torus when $\e_3=\frac14$, i.e.
\begin{equation}\label{4.2}
\Bigl\|\sum_{|\la-\la_j|\le \la^{-\frac14}}E_jf\Bigl\|_{L^6(\Tt)}\le C\la^{\frac34}\|f\|_{L^{\frac65}(\Tt)}, \quad \la\ge 1.
\end{equation}
In order to do this, let us fix an even nonnegative function $a\in {\mathcal S}(\R)$ satisfying $a(0)=1$ and having Fourier transform
supported in $(-1,1)$.  We then claim that to prove \eqref{4.2}, it suffices to see that the operators
\begin{equation}\label{4.3}
a(\la^{\frac14}(\la-\sqrt{-\Delta_{\Tt}}))=\frac1{2\pi}\int_{-\infty}^\infty \la^{-\frac14}\Hat a(t/\la^{\frac14})e^{it\la}e^{-it\sqrt{-\Delta_{\Tt}}}\, dt
\end{equation}
satisfy
\begin{equation}\label{4.4}
\|a(\la^{\frac14}(\la-\sqrt{-\Delta_{\Tt}}))f\|_{L^6(\Tt)}\le C\la^{\frac34}\|f\|_{L^{\frac65}(\Tt)}.
\end{equation}
To see this, we note that if $\sqrt{a}(\tau)$ denotes the square root function $\sqrt{a(\tau)}$, then by a $TT^*$ argument, \eqref{4.4}
would imply that
$$\|\sqrt{a}(\la^{\frac14}(\la-\sqrt{-\Delta_{\Tt}}))f\|_{L^2(\Tt)}\le C\la^{\frac38}\|f\|_{L^{\frac65}(\Tt)}.$$
Since $a(0)=1$, it follows by continuity that $a(\tau)\ge \frac12$ on $[-\delta,\delta]$ for some $\delta>0$,  and, therefore,
by orthogonality, the preceding estimate implies that
$$\Bigl\| \sum_{|\la-\la_j|\le \delta \la^{-\frac14}}E_jf\Bigr\|_{L^2(\Tt)}\le C\la^{\frac38}\|f\|_{L^{\frac65}(\Tt)}.$$
After adding up $O(\delta^{-1})$ of such estimates, we conclude that we may take $\delta=1$.  Consequently, the operator
in \eqref{4.2} maps $L^{\frac65}(\Tt)\to L^2(\Tt)$ with norm the square root of that posited in \eqref{4.2}.  By duality, the
same is true for the $L^2(\Tt)\to L^6(\Tt)$ operator norm, which gives us \eqref{4.2}, since the operator there is a projection operator.

To prove \eqref{4.4}, we note that since the operators $a(\la^{\frac14}(\la+\sqrt{-\Delta_{\Tt}}))$ are trivially bounded
between any Lebesgue spaces with norm $O(\la^{-N})$, it suffices to verify
\begin{equation}\label{4.5}
\Bigl\| \int_{-\infty}^\infty \Hat a(t/\la^{\frac14}) e^{it\la}\cos t\sqrt{-\Delta_{\Tt}}f\, dt\Bigr\|_{L^6(\Tt)}
\le C\la\|f\|_{L^{\frac65}(\Tt)}.
\end{equation}
To do this, we choose an even function $b\in C^\infty_0(\R)$ satisfying
$$b(t)=1, \quad |t|\le1, \quad \text{and } \, \, b(t)=0, \, \, |t|\ge 2.$$
We then split the operator in \eqref{4.5} as $A_0+A_1$, where
$$A_0f=\int_{-\infty}^\infty b(t)\Hat a(t/\la^{\frac14})e^{it\la} \cos t\sqrt{-\Delta_{\Tt}} f\, dt,
$$
and
$$A_1f=\int_{-\infty}^\infty (1-b(t))\Hat a(t/\la^{\frac14})e^{it\la} \cos t\sqrt{-\Delta_{\Tt}} f\, dt.$$
By Lemma~\ref{lemma2.2}, we know that
$$\|A_0f\|_{L^6(\Tt)}\le C\la \|f\|_{L^{\frac65}(\Tt)},$$
since $A_0=\alpha_0(\sqrt{-\Delta_{\Tt}})$, where $|\alpha_0(\tau)|\le C_N((1+|\la-\tau|)^{-N}+(1+|\la+\tau|)^{-N})$ for any $N$.

To finish the proof of \eqref{4.2}, by showing that $A_1$ enjoys similar bounds, by interpolation, it suffices to show that
\begin{equation}\label{4.6}
\|A_1f\|_{L^4(\Tt)}\le C\la^{\frac34}\|f\|_{L^{\frac43}(\Tt)}
\end{equation}
and
\begin{equation}\label{4.7}
\|A_1f\|_{L^\infty(\Tt)}\le C\la^{\frac32}\|f\|_{L^1(\Tt)}.
\end{equation}

To prove the first estimate, we note that $A_1=\alpha_1(\sqrt{-\Delta_{\Tt}})$, where
$$|\alpha_1(\tau)|\le C_N\la^{\frac14}\bigl((1+|\la-\tau|)^{-N}+(1+|\la+\tau|)^{-N}\bigr)$$
for any $N$.  Consequently, \eqref{4.6} follows from the second part of
Lemma~\ref{lemma2.2} as $\sigma(4)=\frac12$ when $n=3$.

\newcommand{\Zt}{{\mathbb Z}^3}

To prove \eqref{4.7}, we need to show that the kernel of $A_1$ is $O(\la^{\frac32})$.  To do this, we
shall use an argument of Hlawka~\cite{Hlawka}.  We first recall that if we identify $\Tt$
with its fundamental domain $Q=(-\frac12, \frac12]^3$, then
$$\bigl(\cos t\sqrt{-\Delta_{\Tt}}\bigr)(x,y)=\sum_{j\in \Zt}\bigl(\cos t\sqrt{-\Delta_{\Rt}})(x-y+j), \quad x,y\in Q,$$
Since, by sharp Huygens principle, $\bigl(\cos t\sqrt{-\Delta_{\Rt}}\bigr)(x)$ is supported on the set 
where $|x|=|t|$, $\Hat a(t)=0$ for $|t|\ge1$ and $(1-b(t))=0$ for $|t|\le 1$ it follows that the kernel $K_1$ of $A_1$ is
$$K_1(x,y)=(2\pi)^{-3}\sum_{\substack {j\in \Zt  \\ |j|\le \la^{\frac14}+1
\\
|x-y+j|\ge 1}}
\int_{-\infty}^\infty \int_{\Rt} e^{i(x-y+j)\cdot \xi}(1-b(t))\Hat a(t/\la^{\frac14}) e^{it\la}\cos t|\xi|\, d\xi dt.
$$
Therefore, in order to show that $|K_1(x,y)|\le C\la^{\frac32}$ and thus obtain \eqref{4.7}, it suffices to show that
for $x,y\in Q$ we have
\begin{equation}\label{4.8}
\sum_{\substack {j\in \Zt  \\ |j|\le \la^{\frac14}+1
\\
|x-y+j|\ge 1}}
\Bigl| \int_{-\infty}^\infty \int_{\Rt}e^{i(x-y+j)\cdot \xi}
(1-b(t))\Hat a(t/\la^{\frac14}) e^{it\la}\cos t|\xi|\, d\xi dt\Bigr|\le C\la^{\frac32}.
\end{equation}

If we replace $(1-b(t))$ by $b(t)$ then, by Huygens principle, all summands with $|j|\ge 10$ vanish since
$b(t)=0$ for $|t|\ge 2$.  If then $\Psi_\la(\tau)$ is the Fourier transform of $t\to b(t)\Hat a(t/\la^{\frac14})$ then,
by Euler's formula, the resulting nonzero summands equal
$$\frac12 \int_{\Rt}e^{i(x-y+j)\cdot \xi}\Bigl(\Psi_\la(\la-|\xi|)+\Psi_\la(\la+|\xi|)\Bigr)\, d\xi,$$
and since the $\Psi_\la$ belong to a bounded subset of ${\mathcal S}(\R)$ for $\la\ge1$, by \eqref{surface}-\eqref{surface2}, each
of these is $O(\la)$, which is better than \eqref{4.8}.  Thus, it suffices to show that if in each summand $(1-b(t))$
is replaced by $1$, then the bound still holds, which is the same thing as showing that for $x,y\in Q$, we have
$$
\sum_{\substack {j\in \Zt  \\ |j|\le \la^{\frac14}+1
\\
|x-y+j|\ge 1}}
\Bigl|\int_{\Rt}e^{i(x-y+j)\cdot \xi}\la^{\frac14}\bigl[a(\la^{\frac14}(\la-|\xi))+a(\la^{\frac14}(\la+|\xi|))\bigr] \, d\xi\Bigr|
\le C\la^{\frac32}.$$
If we use \eqref{surface}-\eqref{surface2} again, as well as the fact that $a\in {\mathbb S}(\R)$, we find that each summand is bounded 
 by
$$\la^{\frac14}\int_0^\infty (1+\la^{\frac14}|\la-r|)^{-3}\, |x+y-j|^{-1}r^{-1}r^2dr
\lesssim \la (1+|j|)^{-1},$$
%by a fixed multiple of
%$\la(1+|j|)^{-1}$, 
which of course implies \eqref{4.8} as
$$\sum_{\{j\in \Zt: \, |j|\le \la^{\frac14}\}}(1+|j|)^{-1}\le C\la^{\frac12}.$$

This completes the proof of \eqref{4.2}.  The same argument works in higher dimensions and yields
\begin{equation}\label{4.9}\Bigl\|\sum_{|\la-\la_j|\le \la^{-\frac1{n+1}}}E_jf\Bigr\|_{L^{\frac{2n}{n-2}}(\Tn)}\le C\la \la^{-\frac1{n+1}}\|f\|_{L^{\frac{2n}{n+2}}(\Tn)},
\end{equation}
i.e., \eqref{4.1} with $\e_n=\frac1{n+1}$.  In this case one would use an interpolation argument
involving $L^{\frac{2(n+1)}{n+3}}\to L^{\frac{2(n+1)}{n-1}}$ and, as before, $L^1\to L^\infty$.  
			%\numberwithin{equation}{subsubsection}

\subsubsection{Improved restriction estimates for $\Tt$}

\par\noindent

In the previous subsection, the special structure of the torus was not exploited since the estimate \eqref{4.6} is merely generic.  On the torus, shrinking spectral 
estimates mean that we need to bound $L^p$-norms of functions like $\sum a_k e^{ik\cdot x}$ where $k$ are lattice points in $\Rn$ which are close to a large
sphere.  Exponential sums like this can be essentially regarded as the Fourier extension of a function with ``porous'' support on such a sphere.  As we shall see
using the techniques of the first author and Guth~\cite{B-G} these restriction-like estimates do not saturate bounds like \eqref{4.6}, and, as a result, further improvements
can be achieved using harmonic analysis.

To
get the improvements  that are stated in Theorem~\ref{theorem1.3}, we shall have to improve the
estimates that arise in the interpolations.  Let us start out by first handling the $3$-dimensional case and then turn to 
the case of $\Tn$, $n\ge4$ in the next subsection.

The most significant improvements over those just obtained earlier (with $\e_3=\frac14$) will come from improving
the $L^{\frac43}(\Tt)\to L^4(\Tt)$ estimates used 
in the interpolation arguments.  To obtain these we first require an estimate of the first author and Guth~\cite{B-G} (see
also \cite{B2}), which involves the Fourier extension operator for the sphere $S^2\subset \Rt$  (or any compact
smooth positively curved hypersurface), which is given by
$$Tf(x)=\int_{S^2}e^{ix\cdot \omega}f(\omega)d\sigma(\omega).$$
The estimate that we require then is the following (see p. 1259 in \cite{B-G})

\begin{lemma}\label{lemma4.1.1}  Let $R\gg 1$ and $\frac1{\sqrt{R}}<\delta_0\le 1$.  For $x\in B_R$, we have
\begin{align}
|Tf| \lesssim_{\e'}\, &R^{\e'}\sum_{\substack {\delta \, \, \text{dyadic}  \\ \delta_0\le \delta \lesssim 1}}
\Bigl[\, \sum_{\tau \, \, \delta-\text{cap}}\bigl( \phi_\tau |Tf_{\tau_1}|^{1/3}|Tf_{\tau_2}|^{1/3}|Tf_{\tau_3}|^{1/3}\bigr)^2\Bigr]^{1/2}
\label{4.1.1}
\\
&+R^{\e'}\Bigl[\sum_{\tau \, \, \delta_0 \, \, \text{cap}}\bigl(\phi_\tau |Tf_\tau|\bigr)^2\Bigr]^{1/2},
\label{4.1.2}
\end{align}
where $f_\tau$ denotes the restriction of $f\in L^2(S^2,d\sigma)$ to $\tau \subset S^2$ and
\begin{equation}\label{4.1.3}
\tau_1, \tau_2,\tau_2\subset \tau \, \, \text{are 3 transversal (in the sense of \cite{B-C-T})} \, \, \frac{\delta}K-\text{caps (with $K$ a large constant)}
\end{equation}
\begin{equation}\label{4.1.4}
\text{For each } \, \tau, \, \phi_\tau\ge 0 \, \, \text{is a function on} \, \, \Rt \, \, \text{satisfying}
\end{equation}
\begin{equation}\label{4.1.05}
\nint_Q \phi_\tau^4 \ll R^{\e'}\end{equation}
for all $Q$ taken in a tiling of $\Rt$ by translates of $\otau$, the polar of the convex hull of $\tau$.  Also, here $\nint_Q$ denotes
the average over $Q$.
\end{lemma}

Let us explain the notation in Lemma~\ref{lemma4.1.1}.
First $B_R$ denotes the ball of radius $R$ centered at the origin.   Also, in \eqref{4.1.1}, and in what follows, $\e'$ denotes an arbitrarily small positive
number and $\lesssim_{\e'}$ denotes an inequality with a constant which depends on the particular choice of 
$\e'$.  
A $\delta$-cap on $S^2$ is just a geodesic ball of radius $\delta$, i.e., a spherical cap of radius $\delta$.  
We
say that three such caps $\tau_1$, $\tau_2$ and $\tau_3$ are transversal if 
the triple wedge product $|x\wedge y \wedge z|$ for any points $x,y,z$ in the respective caps is of order $\delta^2$ up to a factor
that may depend on $K$,
%the angles between each $\tau_j$ and $\tau_k$ is bounded below by a fixed multiple of $\delta$ for all $1\le j<k\le 3$.
%
%
%the absolute value of the matrix with 
%rows $\nu_1$, $\nu_2$ and $\nu_3$ is bounded below by some fixed positive constant, with $\nu_j\in S^2$ being the centers
%of the caps $\tau_j$, $j=1,2,3$.  
Also, for brevity and simplicity of notation, when we write, as in \eqref{4.1.1}, that
we are summing over ``$\delta$ dyadic", we mean that the sum is taken over $\delta=2^{-k}$ with $k\in {\mathbb N}$.
Also, the polar of a $\delta\times \delta\times \delta^2$ rectangle
with short side of length $\delta^2$ pointing in the $\nu\in S^2$ direction is the $\delta^{-1}\times \delta^{-1}\times \delta^{-2}$ rectangle
centered at the origin with long side pointing in the same direction.  Finally, $R\gg 1$ will play the role of $\la$ in \eqref{4.1} or \eqref{4.03}, and we
have decided to use this parameter instead of $\lambda$ since this was the notation used in \cite{B-G}, \cite{B1} and \cite{B2} and we need to
use results from these works.

Since the $|Tf_\tau|$ may be essentially be viewed as constant on the $\otau$-tiles if $K$ is large (see e.g., \cite[Proposition 5.5]{Wolff}) and \eqref{4.1.05} is valid we 
get 

\begin{lemma}\label{lemma4.1.2}  With the above notation
\begin{align}\label{4.1.5}
\|Tf\|_{L^4(B_R)}\lesssim_{\e'} \, &R^{\e'}
\sum_{\substack {\delta \, \, \text{dyadic}  \\ \delta_0\le \delta \lesssim 1}}\Bigl[ \, \sum_{\tau \, \delta-\text{cap}}
\bigl\| \, |Tf_{\tau_1}|^{1/3} \, |Tf_{\tau_2}|^{1/3}\, |Tf_{\tau_3}|^{1/3} \, \bigr\|_{L^4(B_R)}^2\, \Bigr]^{\frac12}
\\
\label{4.1.6}
&+R^{\e'}\Bigl[\, \sum_{\tau\, \delta_0-\text{cap}}\|Tf_\tau\|^2_{L^4(B_R)}\, \Bigr]^{\frac12}.
\end{align}
\end{lemma}

Next, repeating the proof of (4.6) on p. 1260 in \cite{B-G} using parabolic scaling\footnote{By parameterizing the 
sphere $S^3$ appropriately, the phase function in $T$ can be written as $x_1y_1+x_2y_2+x_3\phi(y_1,y_2)$ with $\phi=y_1^2+y_2^2+O(|y|^3)$, and then the 
parabolic scaling is simply $x\to (x_1/\delta,x_2/\delta,x_3/\delta^2)$.}
 and the trilinear $L^3$-inequality from Bennett, Carbery and Tao~\cite{B-C-T} gives
\begin{equation}\label{4.1.7}
\bigl\|\, |Tf_{\tau_1}|^{1/3} |Tf_{\tau_2}|^{1/3}|Tf_{\tau_3}|^{1/3}\, \bigr\|_{L^3(B_R)}\lesssim_{\e'} R^{\e'}\delta^{-\frac13}\|f\|_{L^2}.
\end{equation}
Therefore, by interpolation of $L^4$ between $L^3$ and $L^6$
\begin{multline}\label{4.1.8}
\bigl\|\, |Tf_{\tau_1}|^{1/3} |Tf_{\tau_2}|^{1/3}|Tf_{\tau_3}|^{1/3}\, \bigr\|_{L^4(B_R)}
\\
\lesssim_{\e'} R^{\e'}\delta^{-\frac16}
\Bigl[\, \prod_{i=1}^3\|Tf_{\tau_i}\|_{L^6(B_R)}\, \Bigr]^{1/6}\|f\|_{L^2(d\sigma)}^{\frac12}.
\end{multline}
Fix a subset $\Omega\subset S^2$ (which is a union of $O(\frac1R)$-caps).  For a spherical cap $\tau \subset S^2$, we denote
\begin{equation}\label{4.1.9}
B_p(\tau)=\max \bigl\{\|Tf\|_{L^p(B_R)}: \, \, \text{supp } \, f\subset \tau\cap \Omega \, \, \text{and } \, \, 
\|f\|_{L^2(d\sigma)}\le 1\bigr\},
\end{equation}
and for $\delta>0$, let
\begin{equation}\label{4.1.10}
B_p(\delta)=\max_{\tau \, \delta-\text{cap}}B_p(\tau).
\end{equation}
By choosing functions $f$ with support in $\tau \cap \Omega$ in
 \eqref{4.1.5}, \eqref{4.1.6} and \eqref{4.1.8} we clearly have (noting that $\log R \ll R^{\e'}$) the following

\begin{lemma}\label{lemma4.1.3}  With the above notation, for $\delta_0\le \delta\le 1$,
\begin{equation}\label{4.1.11}
B_4(\delta)\lesssim_{\e'} R^{\e'}B_4(\delta_0)+R^{\e'}\max_{\{\delta_1: \, \delta_0\le \delta_1\le \delta\}}\delta_1^{-\frac16}B_6(\delta_1)^{\frac12}.
\end{equation}
\end{lemma}

As we shall see in the next subsection, similar manipulations are possible in higher dimensions.

To prove estimates for the torus, $\Tt\cong (-\frac12,\frac12]^3$, we need to reformulate \eqref{4.1.11} for the sphere of radius $R$, $RS^2\subset \Rt$
(using the same notation).   So we now fix $\Omega\subset RS^2$ a union of $O(1)$-caps and define for $\rho<R$
\begin{equation}\label{4.1.12}
B_p(\rho)=\max\bigl\{ \, \|Tf\|_{L^p((-\frac12, \frac12]^3)}: \, \, \text{supp } f\subset \tau \cap \Omega, \, \, 
\|f\|_{L^2(RS^2)}\le 1\, \bigr\}.
\end{equation}
Then from Lemma~\ref{lemma4.1.3} and scaling, we obtain the following

\begin{lemma}\label{lemma4.1.4}  With the notation \eqref{4.1.12}, we have for
$\sqrt{R}\le \rho_0<\rho<R$
\begin{equation}\label{4.1.13}
B_4(\rho)\lesssim_{\e'}R^{\e'}B_4(\rho_0)+R^{\frac16 +\e'}\max_{\{\rho_1: \, \rho_0<\rho_1<\rho\}}\rho_1^{-\frac16}B_6(\rho_1)^{\frac12}.
\end{equation}
\end{lemma}

Fix $\frac1{\sqrt{R}}<\e \ll 1$.  Then if $\tau \subset RS^2 $ is a $\rho$-cap, $\rho>\sqrt{R}$, let us set
$$A_\e(\tau)=\{x\in \Rt: \, R\frac{x}{|x|}\in \tau, \, \, |x|\in [R-\e, R+\e \, ]\}.$$
Let us also fix a nonnegative function $\beta\in C^\infty_0(\R)$ with $\beta(s)=1$ for $|s|\le 1/10$ and $\beta(s)=0$ for $|s|>1/4$, and let
\begin{multline}\label{4.1.14}
K^{(\e)}_p(\tau)=K_p(\tau)
\\
=\max\Bigl\{ \, \bigl\|\sum_{k\in \Zt\cap A_\e(\tau)}\beta(\e^{-1}(|k|-R)) \, a_k e^{i k\cdot x}\, \bigr\|_{L^p(\Tt)}: \, \sum_{k\in \Zt}|a_k|^2\le 1\, \Bigr\},
\end{multline}
and
\begin{equation}\label{4.1.15}
K^{(\e)}_p(\rho)=K_p(\rho)=\max_{\tau \, \rho-\text{cap}}K_p(\tau).
\end{equation}
If 
$${\mathcal L}_\e = \{k\in \Zt: |R-|k|\, |<\e\},$$
then if we let $\Omega$ as above be
\begin{equation}\label{4.1.16}
\Omega =\{\, x\in RS^2: \, \, \text{dist }(x,{\mathcal L}_\e)<\frac1{100} \},
\end{equation}
it follows from \eqref{4.1.13} that for $\sqrt{R}<\rho_0<\rho<R$ we have
\begin{equation}\label{4.1.17}
K_4(\rho)\lesssim_{\e'}R^{\e'}K_4(\rho_0)+R^{\e'+\frac16}\, \max_{\{\rho_1: \, \rho_0<\rho_1<\rho\}}\rho_1^{-\frac16}K_6(\rho_1)^{\frac12}.
\end{equation}
To see this, 
we note that if $f$ denotes the restriction to $RS^2$ of $$y\to \sum_{k\in \Zt \cap A_\e(\tau)}\beta(\e^{-1}(|k|-R))\beta(100|y-k|)\, a_k$$ 
then
for $x\in [-\frac12,\frac12]^3$, 
$$Tf(x)=\sum_{k\in \Zt\cap A_\e(\tau)}\alpha(k,x)\beta(\e^{-1}(|k|-R))a_ke^{ik\cdot x},$$ where $|\alpha(k,x)|\approx 1$
and $D^\gamma \alpha(k,x)=O(1)$. Then an easy application of Minkowski's integral
inequality shows that $K_{p}(\rho)\sim B'_{p}(\rho)$ in which $B_{p}'(\rho)$ is a variant
of $B_p(\rho)$ defined as follows
\begin{equation*}
		B_p'(\rho)=\max\Bigl\{ \left\|Tf\right\|_{L^p(
		(-\frac{1}{2},\frac{1}{2}]^3)}, f(y)=\sum_{\gamma
		\in\Omega_{\ast}\cap\tau}\beta_{\gamma}(y-\gamma)a_{\gamma}, y\in RS^2, ||f||_2=1 \Bigr\}.
\end{equation*}
Here $\beta_{\gamma}$ are cut-off functions with $|D^{\alpha}\beta_{\gamma}|\in O(1)$,
$\Omega_{\ast}$ denotes the centers of the caps in $\Omega$. By (4.2.6), (4.2.7), (4.2.9) and scaling $B_p'(\rho)$ also
satisfies \eqref{4.1.13}.

To estimate $K_p(\rho)$ it is convenient to involve a comparable quantity involving smoothed out multipliers.  Specifically, 
if $\rho \subset RS^2$ is a $\rho$-cap with center $\omega_\tau\in RS^2$, let us define for $f\in L^2(\Tt)$
\begin{equation}\label{4.1.170}
T^\e_\tau f = \sum_{k\in \Zt}\beta(\e^{-1}(|k|-R)) \beta(\rho^{-1}|k-\omega_\tau|) \Hat f(k)e^{2\pi ik\cdot x}.
\end{equation}
If we then put for $p>2$
$$\tilde K^{(\e)}_p(\rho)=\tilde K_p(\rho)=
\max_{\tau \, \rho-\text{cap}}\|T^\e_\tau\|_{L^2(\Tt)\to L^p(\Tt)},$$
then the simple orthogonality arguments used before yield
\begin{equation}\label{4.1.171}
K_p(\rho)\approx \tilde K_p(\rho).
\end{equation}

If 
$$m_{\tau,\e}(\xi)=\bigl(\, \beta(\e^{-1}(|\xi|-R))\beta(\rho^{-1}(|\xi-\omega_\tau|))\, \bigr)^2,$$
and we let
$$Mf(x)=\sum_{k\in \Zt}m_{\tau,\e}(k)\Hat f(k)e^{2\pi i k\cdot x},$$
then of course
\begin{equation}\label{4.1.20}
(\tilde K_p(\tau))^2=\|M\|_{L^{p'}(\Tt)\to L^p(\Tt)},
\end{equation}
where $p'=p/(p-1)$.  To estimate $M$, we split it into two parts in a similar way to how the operator in \eqref{4.3} was split into
$A_0+A_1$.  In this case, let us fix a $C^\infty_0(\Rt)$ bump function $\eta$ supported in the unit ball satisfying $\int_{\Rt}\eta = 1$.
We then put
$$M_0f =\sum_{k\in \Zt}m^0_{\tau,\e}(k)\Hat f(k)e^{2\pi i k\cdot x},$$
where
$$m^0_{\tau,\e}(\xi)=\bigl(m_{\tau,\e}*\eta\bigr)(\xi).$$
We then can estimate the $L^{\frac65}(\Tt)\to L^6(\Tt)$ norm of $M$ as follows
\begin{equation}\label{4.1.22}
\|M\|_{\frac65\to 6}\le \|M_0\|_{\frac65\to 6}
+\bigl(\|M\|_{\frac43\to 4}+\|M_0\|_{\frac43\to 4}\bigr)^{\frac23}\|M-M_0\|_{1\to \infty}^{\frac13}.
\end{equation}
From Lemma~\ref{lemma2.2} we deduce that
\begin{equation}\label{4.1.23}
\|M_0\|_{\frac43\to 4}\le \e R^{1/2},
\end{equation}
and interpolation with $\|M_0\|_{1\to\infty}\lesssim \e\rho^2$ gives
\begin{equation}\label{4.1.24}
\|M_0\|_{\frac65\to 6}\lesssim \e R^{\frac13}\rho^{\frac23}.
\end{equation}

To evaluate the last factor in \eqref{4.1.22} we note that
$$\|M-M_0\|_{1\to \infty}=\bigl\| \, \sum_{k\in \Zt}(m_{\tau,\e}(k)-m^0_{\tau,\e}(k))e^{2\pi ik\cdot x}\bigr\|_{L^\infty(\Tt)}.$$
Taking $x\in [-\frac12,\frac12]^3$, Poisson summation gives
$$\left| \sum_{k\in \Zt}(m_{\tau,\e}(k)-m^0_{\tau,\e}(k))e^{2\pi k\cdot x}\right|
=\left|\sum_{k\in \Zt}\bigl( \Hat m_{\tau,\e}(1-\Hat \eta)\bigr)(k+x)\right|,$$
where here for $g\in {\mathcal S}(\Rt)$, $\Hat g(\xi)=\int_{\Rt}e^{-2\pi i x\cdot \xi}g(x)\, dx$.
Consequently,
\begin{multline}\label{4.1.25}
\|M-M_0\|_{1\to \infty}
\\
\le \sup_{x\in [-\frac12,\frac12]^3}|\Hat m_{\tau,\e}(x)|\, |1-\Hat \eta(x)| +\sup_{x\in [-\frac12,\frac12]^3}\sum_{|k|>0}
|(\Hat m_{\tau,\e}\Hat \eta)(k+x)|
\\
+\sup_{x\in [-\frac12,\frac12]^3}\sum_{|k|>0}|\Hat m_{\tau,\e}(k+x)|.
\end{multline}
Recall that $\rho>\sqrt{R}$, therefore, by
using stationary phase (for example, Theorem 7.7.6 in \cite{Hor3}), 
%through a careful
%calculation 
one finds that for every $N=1,2,3,\dots$
\begin{equation}\label{4.1.28}
|\Hat m_{\tau,\e}(y)|\le C_N \frac{\e R}{|y|} \, \bigl(1+\e|y|)^{-N},
\end{equation}
which implies that the first two terms in the right side of \eqref{4.1.25} are $O(\e R)$,
since $\Hat \eta(0)=1$.  To estimate the third term in the right side we note that
there is a constant $A$ so that if $C$ is the conic region with vertex at the origin and
 axis passing through the center $\omega_\tau$ of $\tau$
and angle $A\frac\rho{R}$, then for every $N=1,2,3,\dots$ the stationary phase argument we
used to prove \eqref{4.1.28} also shows that 
\begin{equation}\label{4.1.280}|\Hat m_{\tau,\e}(y)|\le C_N \e \rho^2
 \,  \Bigl(1+\frac{\rho^2}{R}|y|\Bigr)^{-N}, \quad x\notin C,
\end{equation}
which, together with \eqref{4.1.28} means that $\Hat m_{\tau,\e}$ is essentially supported in
$B_{O(\frac1\e)}\cap(C \cup B_{O(\frac{R}{\rho^2})})$.
Using these two bounds we find the third term in the right side of \eqref{4.1.25} is majorized by
$$\e R\sum_{0\ne k\in C} |k|^{-1}(1+\e|k|)^{-3}+\e \rho^2\sum_{k\ne 0}(1+\rho^2|k|/R)^{-4}
\lesssim \e R \Bigl(\frac\rho{\e R}\Bigr)^2+\e R.$$
If we combine this with our earlier estimates for the first two terms in the right side of \eqref{4.1.25}, 
we conclude that
\begin{equation}\label{4.1.30}
\|M-M_0\|_{1\to \infty}\lesssim \e R+\frac{\rho^2}{\e R}.
\end{equation}
Therfore, by \eqref{4.1.23}, \eqref{4.1.24} and \eqref{4.1.22}, we have
$$\|M\|_{\frac65\to 6}\lesssim \e R^{\frac13}\rho^{\frac23}+
\bigl(\|M\|_{\frac43\to 4}+\e R^{\frac12}\bigr)^{\frac23}\, \bigl(\e R+\frac{\rho^2}{\e R}\bigr)^{\frac13},$$
implying by \eqref{4.1.171} and \eqref{4.1.20} that
\begin{equation}\label{4.1.31}
K_6(\rho)\lesssim \sqrt{\e}R^{\frac16}\rho^{\frac13}
+\bigl(K_4(\rho)+\sqrt{\e}R^{\frac14}\bigr)^{\frac23}\, \bigl(\e R+\frac{\rho^2}{\e R}\bigr)^{\frac16}.
\end{equation}
Also, since $\|M-M_0\|_{2\to 2}=O(1)$, by using \eqref{4.1.23} and \eqref{4.1.30} we find
$$\|M\|_{\frac43\to4}\le \|M_0\|_{\frac43\to4}+\|M-M_0\|_{1\to\infty}^{\frac12}
\lesssim \e R^{\frac12}+\bigl(\e^{\frac12}R^{\frac12}+\frac\rho{(\e R)^{\frac12}}\bigr),$$
and, therefore by \eqref{4.1.171},
\begin{equation}\label{4.1.32}
K_4(\rho)\lesssim \e^{\frac12}R^{\frac14}+\bigl(\e^{\frac14}R^{\frac14}+\frac{\rho^{\frac12}}
{(\e R)^{\frac14}}\bigr).
\end{equation}
Set $\rho_0=(\frac{R}\e)^{\frac12}$.  From \eqref{4.1.32}, we obtain
\begin{equation}\label{4.1.33}
K_4(\rho)\lesssim \e^{\frac14}R^{\frac14+} \quad \text{for } \, \, \rho\le \rho_0,
\end{equation}
provided that
\begin{equation}\label{4.1.34}
\e>R^{-\frac13}.
\end{equation}
In order to prove that \eqref{4.1.33} holds for all $\rho$, we combine
\eqref{4.1.17} and \eqref{4.1.31}.  Denoting $K_4\equiv K_4^{(\e)}(R)$, it follows that
\begin{multline}\label{4.1.35}
K_4\lesssim \e^{\frac14}R^{\frac14+}+R^{\frac16+}
\max_{\rho_1>\rho_0} \rho_1^{-\frac16}
\Bigl\{ \, \e^{\frac14}R^{\frac1{12}}\rho_1^{\frac16}+K_4^{\frac13}\bigl(\e^{\frac1{12}}R^{\frac1{12}}
+\frac{\rho_1^{\frac16}}{(\e R)^{\frac1{12}}}\, \bigr)\Bigr\}
\\ 
\lesssim \e^{\frac14}R^{\frac14+}+K_4^{\frac13}\Bigl((\e R)^{\frac16+}+
\bigl(\frac{R}\e\bigr)^{\frac1{12}+}\Bigr),
\end{multline}
and hence, by \eqref{4.1.34}, 
\begin{equation}\label{4.1.36}
K_4\lesssim \e^{\frac14}R^{\frac14+},
\end{equation}
as claimed.

Since $K_4=K^{(\e)}_4$ is obviously an increasing function of $\e$ we have the following

\begin{lemma}\label{lemma4.1.5}
	\begin{align}
		K_4^{(\varepsilon)}&\lesssim\varepsilon^{\frac{1}{4}}R^{\frac{1}{4}+}\quad
		\textrm{ if } \varepsilon\geq R^{-\frac{1}{3}}\label{4.1.37}\\
                K_4^{(\varepsilon)}&\lesssim R^{\frac{1}{6}+}\qquad
		\textrm{ if } \varepsilon\leq R^{-\frac{1}{3}}\label{4.1.38}.
	\end{align}
\end{lemma}

\newcommand{\expo}{e^{ix\cdot\xi}}

{\bf\noindent Remarks:}
\begin{enumerate}
	\item[(i)] Considering $R\in\mathbb{Z}$ and the points
		\begin{equation*}
			\left\{ (z_1,z_2,R)\in\Zt:\, 
			\max(|z_1|,|z_2|)<\frac{1}{10}\sqrt{R\varepsilon}
			\right\}\subset {\mathcal L}_{\varepsilon}
		\end{equation*}
		shows that (\ref{4.1.34}) is essentially optimal.
	\item[(ii)] Possibly (\ref{4.1.37}) holds for all $\varepsilon>\frac{1}{R}$. Note that if
		we assume $R^2\in\mathbb{Z}$ %a prime, 
		\begin{equation}
			K^{(\frac{1}{R})}_4\sim\max\left\{
		\bigl\| \sum_{|\xi|^2=R^2}a_{\xi}\expo \bigr\|_4, \, \, \sum|a_{\xi}|^2\leq 1 \right\}\ll
		R^{0+},
			\label{4.1.39}
		\end{equation}
		%since for $z\in B_R$ we may bound the number of representations
		since for $z\in B_R$ the number of representations
%		\begin{equation*}\begin{split}
%			&|\left\{ (\xi,\eta)\in\Zt\times\Zt; |\xi|=|\eta|=R, \xi+\eta=2z
%			\right\}|\\\leq&|\left\{ (\xi,\eta)\in\Zt\times\Zt;
%			|\xi|^2=|\eta|^2=R^2, \xi\cdot\eta=2|z|^2-R^2 \right\}|\ll
%			R^{0+}\end{split}
%		\end{equation*}
$$\left\{ (\xi,\eta)\in\Zt\times\Zt; |\xi|=|\eta|=R, \xi+\eta=2z
			\right\}$$
		amounts to the number of lattice points in a planar section of the $R$-sphere, which is bounded by $R^\varepsilon$ (see \S 2 in \cite{BourRud}).
\end{enumerate}

\bigskip

Applying \eqref{4.1.22} with $\tau = RS^2$ it follows from Lemma~\ref{lemma4.1.5},
\eqref{4.1.23}, \eqref{4.1.24} and \eqref{4.1.30} that
if we set $$Mf=M_\e f=\sum_{k\in \Zt}(\beta(\e^{-1}(|k|-R)))^2\Hat f(k)e^{ik\cdot x},$$ then
\begin{multline}\label{4.1.40}
\|M_\e\|_{\frac65\to 6}\lesssim \e R+R^{0+}\bigl(\e^{\frac13}R^{\frac13}+R^{\frac29}\bigr)
\|M-M_0\|_{1\to \infty}^{\frac13}\\
\lesssim \e R+
\bigl(\e^{\frac13}R^{\frac13}+R^{\frac29}\bigr) \frac{R^{\frac13+}}{\e^{\frac13}}
\lesssim \e R,
\end{multline}
provided that $\e>R^{-\frac13+}$.  Consequently, we have \eqref{4.03} for $n=3$ with
$\e(\la)=\la^{-\frac13+}$, which improves our earlier results  of $\la^{-\frac14}$.

We can make a further improvement of what was done in the last step by a less rough
evaluation of $\|M-M_0\|_{1\to \infty}$.  Fixing $x\in [-\frac12, \frac12]^3$, Poisson summation
leads to bounding exponential sums of the type
\begin{equation}\label{4.1.41}
\e R \sum_{j\in \Zt, \, 0<|j|\lesssim \frac1\e}\frac{e^{iR|j+x|}}{|j+x|}.
\end{equation}
For simplicity, we are ignoring the tails in the sum where $ 1/\e \lesssim |j|$.  However, by (4.2.26), these are rapidly
decreasing when $|j|\gg 1/\e$, and so the argument bounding \eqref{4.1.41} easily controls the corresponding sum
over $1/\e \lesssim |j|$.

Breaking up summation ranges in \eqref{4.1.41}, we need to evaluate for $M\lesssim \frac1\e$
\begin{equation}\label{4.1.42}
\frac1M \Bigl| \sum_{j_1\in I_1, \, j_2\in I_2, \, j_3\in I_3}e^{iR|j+x|}\Bigr|
\end{equation}
with $I_1\times I_2\times I_3 \subset B_M\backslash B_{\frac{M}2}$ a cube of size $\frac{M}{10}$.
Hence, fixing one  of the variables, say, $j_3$, and making a coordinate change
$(j_1,j_2)\to (j_1+j_2,j_2)$ or $(j_1,j_2)\to (j_1, j_2+j_1)$ if needed, we get
\begin{equation}\label{4.1.43}
\Bigl|\sum_{(j_1,j_2)\in {\mathcal D}}e^{if(j_1,j_2)}\Bigr|
\end{equation}
where ${\mathcal D}$ is a quadrangle such that $|j_1|, |j_2|\approx M$ for $(j_1,j_2)\in {\mathcal D}$
and $f(\alpha,\beta)$ is one of the functions
\begin{align*}
	&R[(x_1+\alpha)^2+(x_2+\beta)^2+(x_3+j_3)^2]^{\frac{1}{2}}\\
	&R[(x_1+\alpha-\beta)^2+(x_2+\beta)^2+(x_3+j_3)^2]^{\frac{1}{2}}
\end{align*}
or
\begin{equation*}
	R[(x_1+\alpha)^2+(x_2+\beta-\alpha)^2+(x_3+j_3)^2]^{\frac{1}{2}}.
\end{equation*}

At this point, we may invoke the exponential sum bound (stated as Lemma 2) in
the paper \cite{K-N}, the above function $f$ satisfying the required conditions (with
$\bigwedge=R$ and any $k$).

It follows that
\begin{equation}
	(\ref{4.1.42})\lesssim\log{R}\, M^2\Bigl(\frac{R}{M^{k+1}}\Bigr)^{\frac{1}{32^k-2}}\quad
	\textrm{ for any } k\in\mathbb{Z}_{+}.
	\label{4.1.44}
\end{equation}
Since (\ref{4.1.44}) is increasing in $M$, it follows that
\begin{equation}
	(\ref{4.1.41})\lesssim
	R\log{R}\, \frac{1}{\varepsilon}(R\varepsilon^{k+1})^{\frac{1}{32^k-2}}\quad 
	\textrm{ for any } k\in \mathbb{Z}_{+}
	\label{4.1.45}
\end{equation}
which bounds $||M-M_0||_{1\to\infty}$.

Taking $k=3$, substitution in $(\ref{4.1.40})$ gives
\begin{equation}
	||M_{\varepsilon}||_{\frac{6}{5}\to6}\lesssim\varepsilon
	R+R^{\frac{2}{9}+\frac{1}{3}+\frac{1}{66}+}\varepsilon^{-\frac{1}{3}+\frac{2}{33}}\lesssim\varepsilon
	R\quad \textrm{ for } \varepsilon>R^{-\frac{85}{252}},
	\label{4.1.46}
	\end{equation}
hence we have \eqref{4.03} for any $\e_3<\frac{85}{252}$
\begin{equation*}
\Bigl\| \sum_{\{k\in \Zt: \, |\, |k|- R|\le \e)\}}\Hat f(j)e^{2\pi ik\cdot x}\Bigr\|_{L^{6}(\Tt)}
\le C\e R \|f\|_{L^{\frac{6}{t}}(\Tt)}, \, \,  \text{for } \, \e>R^{-\frac{85}{252}+}, \, \frac{85}{252}=0.337...,
\end{equation*}
which is the first part of Theorem~\ref{theorem1.3}.

\bigskip

\subsubsection{Improved restriction estimates for $\Tn$, $n>3$}

\par\noindent

Let $n>3$.  We require the following result which follows from the arguments in \S 3 and \S 4 of \cite{B-G} which is a higher dimensional
version of Lemma~\ref{lemma4.1.1}

\begin{lemma}
	\label{lemma4.3.1}Fix $3\leq k\leq n$. Let $R\gg1$ and
	$\frac{1}{\sqrt{R}}<\delta_0<1$. On $B_R$, we have
	\begin{align}
		|Tf|\lesssim_{\varepsilon '}R^{\varepsilon '}&\sum_{\substack{\delta\ \rm{
		dyadic}\\
		\delta_0<\delta\lesssim1}}\left[ \sum_{\tau\
		\delta-cap}\Bigl(\phi_{\tau}\prod_{j=1}^k|Tf_{\tau_j}|^{\frac{1}{k}}\Bigr)^2
		\right]^{\frac{1}{2}} \label{N1}\\
		&+R^{\varepsilon '}\left[\sum_{\tau\ \delta_0-cap}\bigl(\phi_{\tau}|Tf_{\tau}|\bigr)^2\right]^{\frac{1}{2}}\label{N2}
	\end{align}
	for $f\in L^2(S^{n-1},d\sigma)$, where
\begin{enumerate}
		\item[(i)] $\tau_1,\dots,\tau_k\subset\tau$ are k-transversal
			caps of size $\rho$;
		\item[(ii)] For each $\tau,\phi_{\tau}\geq0$ is a function on $\R^ n$
			satisfying
			\begin{equation}
				\nint_Q\phi_{\tau}^{\frac{2(k-1)}{k-2}}\ll R^{\varepsilon},
				\label{N5}
			\end{equation}
			for all $Q$ taken in a tiling of $\R^ n$ by translates of $\otau$.
	\end{enumerate}
\end{lemma}

We choose $k$ such that
\begin{equation*}
	\frac{2(k-1)}{k-2}\geq\frac{2(n+1)}{n-1}>\frac{2k}{k-1}
\end{equation*}
thus
\begin{equation}
	\frac{n+1}{2}<k\leq\frac{n+3}{2}.
	\label{n6}
\end{equation}
We may then state the analogue of Lemma \ref{lemma4.1.2}, denoting
$p_0=\frac{2(n+1)}{n-1}, p_1=\frac{2n}{n-2}$, then
\begin{lemma}
	\label{lemma4.3.2}
	\begin{align}
		||Tf||_{L^{p_0}(B_R)}\lesssim_{\e'} R^{\varepsilon'}\sum_{\substack{\delta\
		\rm{dyadic}\\ {\delta_0<\delta<1}}}&\left[ \sum_{\tau\
		\delta-cap}\Bigl\| \, \prod_{j=1}^k|Tf_{\tau_j}|^{\frac{1}{k}}
		\Bigr\|^2_{L^{p_0}({B_R})}\right]^{\frac{1}{2}}\label{N7}\\
		&+R^{\varepsilon'}\left[ \sum_{\tau\
		\delta_0-cap}\bigl\|Tf_{\tau}\bigr\|^2_{L^{p_0}({B_R)}}
		\right]^{\frac{1}{2}}. \label{N8}
	\end{align}
\end{lemma}
From parabolic rescaling and the $k$-linear inequality from \cite{B-C-T} we have
\begin{equation}
	\bigl\|\prod_{j=1}^{k}|Tf_{\tau_j}|^{\frac{1}{k}}\bigr\|_{L^{\frac{2k}{k-1}}(B_r)}\lesssim_{\e'}
	R^{\varepsilon'}\delta^{\frac{n+1}{2k}-1}.
	\label{N(}
\end{equation}
Interpolating $L^{p_0}$ between $L^{\frac{2k}{k-1}}$ and $L^{p_1}$ gives
\begin{equation}
	\bigl\|\prod_{j=1}^k|Tf_{\tau_j}|^{\frac{1}{k}}\bigr\|_{L^{p_0}(B_R)}\lesssim_{\e'}
	R^{\varepsilon'}\delta^{-\frac{\theta}{n}}\prod_{j=1}^k\bigl\|Tf_{\tau_j}\bigr\|_{L^{p_1}(B_R)}^{\frac{\theta}{k}}
	\label{N10}
\end{equation}
with 
\begin{equation}
	\theta=\frac{\frac{2}{n+1}-\frac{1}{k}}{\frac{2}{n}-\frac{1}{k}}.
	\label{N11}
\end{equation}
Using same notation, Lemma \ref{lemma4.1.3} becomes
\begin{lemma}
	\label{lemma8}
	\begin{equation}
		B_{p_0}(\delta)\lesssim
		R^{0+}B_{p_0}(\delta_0)+R^{0+}\max_{\delta_0<\delta_1<\delta}\delta_1^{-\frac{\theta}{n}}B_{p_1}(\delta_1)^{\theta}.
		\label{N12}
	\end{equation}
\end{lemma}
Rescaling, we get for caps on $RS^{n-1}$
\begin{lemma}
	\label{lemma9}
	\begin{equation}
		B_{p_0}(\rho)\lesssim
		R^{0+}B_{p_0}(\rho_0)+R^{\frac{1}{p_0}-\theta(\frac{1}{2}-\frac{1}{n})+}\max_{\rho_0<\rho_1<\rho}\rho_1^{-\frac{\theta}{n}}B_{p_1}(\rho_1)^{\theta}
		\label{N13}
	\end{equation}
\end{lemma}
Inequality (\ref{4.1.17}) becomes
\begin{equation}
	K_{p_0}(\rho)\lesssim
	R^{0+}K_{p_0}(\rho_0)+R^{\frac{1}{p_0}-\theta(\frac{1}{2}-\frac{1}{n})+}\max_{\rho_0<\rho_1<\rho}\rho_1^{-\frac{\theta}{n}}K_{p_1}(\rho_1)^{\theta}.
	\label{N14}
\end{equation}

Let us distinguish between the cases where $n$ is odd and even.

\newcommand{\cL}{\mathcal{L}}

\begin{itemize}
	\item n is odd. Then $k=\frac{n+3}{2}$,
		$\theta=\frac{2}{3}\cdot\frac{n}{n+1}$ and (\ref{N14}) gives
		\begin{equation}
			K_{p_0}(\rho)\lesssim
			R^{0+}K_{p_0}(\rho_0)+R^{\frac{1}{6}+}\max_{\rho_0<\rho_1<\rho}\rho_1^{-\frac{2}{3(n+1)}}K_{p_1}(\rho_1)^\frac{2n}{3(n+1)}.
			\label{N15}
		\end{equation}
	\item n is even. Then $k=\frac{n}{2}+1$, $\theta=\frac{n}{2(n+1)}$ and 
		\begin{equation}
			K_{p_0}(\rho)\lesssim
			R^{0+}K_{p_0}(\rho_0)+R^{\frac{n}{4(n+1)}}\max_{\rho_0<\rho_1<\rho}\rho_1^{-\frac{1}{2(n+1)}}K_{p_1}(\rho_1)^{\frac{n}{2(n+1)}}.
			\label{N16}
		\end{equation}
\end{itemize}
We will use Proposition 1 from \cite{B1}, which gives, after rescaling
\begin{lemma}
	\label{lemma10}
	Let $1\ll\rho<\sqrt{R}$ and $\left\{ \tau_{\alpha}(\rho) \right\}$ a partition of
	$RS^{n-1}$ in cells of size $\rho$. For $r>C(x)\frac{R}{\rho^2}$ and
	$q=\frac{2n}{n-1}$, one has the inequality
	\begin{equation}
		\Bigl\|\int
		g(\xi)e^{ix\cdot\xi}\sigma_R(d\xi)\Bigr\|_{L^q_{(r)}}\lesssim_{\kappa}R^{\kappa}\left\{
		\sum_{\alpha}\Bigl\|\int_{\tau_{\alpha}(\rho)}g(\xi)e^{ix\cdot\xi}\sigma_R(d\xi)\Bigr\|^2_{L^q_{(r)}}
		\right\}^{\frac{1}{2}}
		\label{N17}
	\end{equation}
	where $L^q_{(r)}$ denotes $L^q(\eta(\frac{x}{r}),dx),$ where $0\leq\eta \le1$ is some
	rapidly decreading bump function on $\R^ n$. Also $\kappa>0$ is an arbitrarily small,
	fixed constant. 
\end{lemma}
Exploiting Lemma \ref{lemma10} in the content of lattice points requires the following
simple observation that was also crucial in \cite{B1}

Let $\rho=\varepsilon^{\frac{1}{2}}R^{\frac{1}{2}+}$ in Lemma \ref{lemma10}, so that we
may take $r=o(\frac{1}{\varepsilon})$. Let $\cL_{\varepsilon}'\subset RS^{n-1}$ be a
collection of points obtained by $\varepsilon$-perturbation of the points of
$\cL_{\varepsilon}\subset\Zn$. Applying (\ref{N17}) in a discretized
version gives
\begin{equation}
	\Bigl\|\sum_{\xi\in\cL_{\varepsilon}'}a_{\xi}e^{ix\cdot\xi'}\Bigr\|_{L^q_{(r)}}\lesssim
	R^{0+}\left[
	\sum_{\alpha}\, \Bigl\|\sum_{\xi\in\cL_{\varepsilon}'(\tau_{\alpha})}a_{\xi}e^{ix\cdot\xi'}\Bigr\|_{L^q_{(r)}}^{2}
	\right]^{\frac{1}{2}}
	\label{N18}
\end{equation}
and since $|\xi-\xi'|<\varepsilon$ and $r\ll1/\varepsilon$, a perturbation argument
permits us to deduce from (\ref{N18}) that we also have
\begin{equation}
\Bigl\|\sum_{\xi\in\cL_{\varepsilon}}a_{\xi}e^{ix\cdot\xi}\Bigr\|_{L^q_{(r)}}\lesssim
	R^{0+}\left[
	\sum_{\alpha}\, \bigl\|\sum_{\xi\in\cL_{\varepsilon}(\tau_{\alpha})}b_{\xi}e^{ix\cdot\xi}\bigr\|_{L^q_{(r)}}^{2}
	\right]^{\frac{1}{2}}
	\label{N19}
\end{equation}
for some coefficients $\left\{ b_{\xi} \right\}, |b_{\xi}|\leq|a_{\xi}|$. 

Now, since the functions $\sum_{\xi\in\cL_{\varepsilon}}a_{x}\expo$ and
$\sum_{\xi\in\cL_{\varepsilon}}b_{\xi}\expo$ are $1$-periodic, (\ref{N19}) is equivalent
to
\begin{equation}
\bigl\|\sum_{\xi\in\cL_{\varepsilon}}a_{\xi}e^{ix\cdot\xi}\bigr\|_{L^q(\Tn)}\lesssim
	R^{0+}\left[
	\sum_{\alpha}\, \bigl\|\sum_{\xi\in\cL_{\varepsilon}(\tau_{\alpha})}b_{\xi}e^{ix\cdot\xi}\bigr\|_{L^q(\Tn)}^{2}
	\right]^{\frac{1}{2}}.
	\label{N20}
\end{equation}
Hence
\begin{equation*}
	K_q^{(\varepsilon)}(R)\lesssim R^{0+}K_{q}^{(\varepsilon)}(R^{0+}\sqrt{\varepsilon R})
\end{equation*}
and therefore we have the following
\begin{lemma}
	\label{lemma11}
	\begin{equation}
		K_q^{(\varepsilon)}(R)\lesssim (\varepsilon
		R)^{\frac{n-1}{4n}+},\quad\textrm{where } \quad q=\frac{2n}{n-1}.
		\label{N21}
	\end{equation}
\end{lemma}

Turning  to the second part of the argument, (\ref{4.1.22}) gets replaced by 
\begin{equation}
	||M||_{p_1'\to p_1}\leq||M_0||_{p_1'\to p_1}+(||M||_{p_0'\to
	p_0}+||M_0||_{p_0'\to
	p_0})^{1-\frac{2}{n(n-1)}}||M-M_0||_{1\to\infty}^{\frac{2}{n(n-1)}}
	\label{N22}
\end{equation}
and interpolation of $L^{p_0}$ between $L^q$ and $L^{\infty}$ together with
(\ref{N21}) implies
\begin{equation}
	||M||_{p_0'\to p_0}\leq||M_0||_{p_0'\to
	p_0}+||M-M_{0}||_{1\to\infty}^{\frac{1}{n+1}}( (\varepsilon
	R)^{\frac{n-1}{2n}}+||M_0||_{q'\to q})^{\frac{n}{n+1}}.
	\label{N23}
\end{equation}

Next, we have
\begin{equation}
	||M_0||_{p_0'\to p_0}\lesssim\varepsilon R^{\frac{n-1}{n+1}},\quad
	||M_0||_{q'\to q}\lesssim\varepsilon R^{\frac{n-1}{2n}}
	\label{N24}
\end{equation}
and interpolating with $L^{\infty}$ gives
\begin{equation}
	||M_0||_{p_1'\to p_1}\lesssim\varepsilon
	R^{1-\frac{2}{n}}\rho^{\frac{2}{n}}.
	\label{N25}
\end{equation}

Similarly, the $n$-dimensional version of \eqref{4.1.28} is
$$|\Hat m_{\tau,\e}(y)|\le C_N\frac{\e R^{\frac{n-1}2}}{|y|^{\frac{n-1}2}}(1+\e|y|)^{-N}, \quad N=1,2,3,\dots,$$
and, also, as before we get better estimates if $y$ is not in a conic region $C$ centered at the center of $\tau$
and of angle $O(\frac\rho{R})$, namely,
$$|\Hat m_{\tau,\e}(y)|\le C_N\e \rho^{n-1}\bigl(1+\frac{\rho^2}R |y|\bigr)^{-N},
\, \, y\notin C, \, \, N=1,2,3,\dots .$$
Therefore, the earlier arguments involving the Poisson summation formula yield
\begin{equation}
	||M-M_0||_{1\to\infty}\lesssim\varepsilon
	R^{\frac{n-1}{2}}+\Bigl(\frac{\rho^2}{\varepsilon R}\Bigr)^{\frac{n-1}{2}}.
	\label{N28}
\end{equation}

It follows from (\ref{N23}), (\ref{N24}), (\ref{N28}) that
\begin{equation}
	||M||_{p_0'\to p_0}\lesssim\varepsilon R^{\frac{n-1}{n+1}}+\left[ \varepsilon
	R^{\frac{n-1}{2}}+\Bigl(\frac{\rho^2}{\varepsilon R}\Bigr)^{\frac{n-1}{2}}
	\right]^{\frac{1}{n+1}}(\varepsilon R)^{\frac{n-1}{2(n+1)}}
	\label{N29}
\end{equation}
and from (\ref{N22}), (\ref{N24}), (\ref{N25}), (\ref{N28})
\begin{equation}
	||M||_{p_1'\to p_1}\lesssim\varepsilon
	R^{1-\frac{2}{n}}\rho^{\frac{2}{n}}+(\varepsilon
	R^{\frac{n-1}{n+1}}+||M||_{p_0'\to
	p_0})^{1-\frac{2}{n(n-1)}}(\varepsilon^{\frac{2}{n(n-1)}}R^{\frac{1}{n}}+(\frac{\rho^2}{\varepsilon
	R})^{\frac{1}{n}}).
	\label{N30}
\end{equation}
Hence
\begin{equation}
	K_{p_1}(\rho)\lesssim\sqrt{\varepsilon}R^{\frac{1}{2}-\frac{1}{n}}\rho^{\frac{1}{n}}+(\varepsilon^{\frac{1}{2}}R^{\frac{n-1}{2(n+1)}}+K_{p_0}(\rho))^{1-\frac{2}{n(n-1)}}(\varepsilon^{\frac{1}{n(n-1)}}R^{\frac{1}{2n}}+(\frac{\rho}{\sqrt{\varepsilon
	R}})^{\frac{1}{n}}).
	\label{N31}
\end{equation}

Let
\begin{equation}
	\rho_0=\varepsilon^{\frac{n+1}{2(n-1)}}R.
	\label{N32}
\end{equation}
It follows from (\ref{N29}) that for $\rho\leq\rho_0$
\begin{equation}
	K_{p_0}(\rho)\lesssim\varepsilon^{\frac{1}{4}}R^{\frac{n-1}{2(n+1)}+}.
	\label{N33}
\end{equation}

\noindent\underline{Case of $n$ odd}

From (\ref{N33}), (\ref{N15}) and (\ref{N31}) we have
\begin{equation*}
	\begin{split}
		K_{p_0}&\lesssim\varepsilon^{\frac{1}{4}}R^{\frac{n-1}{2(n+1)}+}+R^{\frac{1}{6}+}\max_{\rho_0<\rho_1}\rho_1^{-\frac{2}{3(n+1)}}K_{p_1}(\rho_1)^{\frac{2n}{3(n+1)}}\\
		&\lesssim\varepsilon^{\frac{1}{4}}R^{\frac{n-1}{2(n+1)}+}+\varepsilon^{\frac{n}{3(n+1)}}
		R^{\frac{n-1}{2(n+1)}+}
		+R^{\frac{1}{6}+}\max_{\rho_1>\rho_0}\rho_1^{-\frac{2}{3(n+1)}}K_{p_0}^{(1-\frac{2}{n(n-1)})\frac{2n}{3(n+1)}}(\frac{\rho_1^2}{\varepsilon
		R})^{\frac{1}{3(n+1)}}\\
		&\lesssim\varepsilon^{\frac{1}{4}}R^{\frac{n-1}{2(n+1)}+}+R^{\frac{1}{6}-\frac{1}{3(n+1)}}\varepsilon^{-\frac{1}{3(n+1)}}K_{p_0}^{(1-\frac{2}{n(n-1)})\frac{2n}{3(n+1)}}
	\end{split}
\end{equation*}
from which we deduce that
\begin{equation}
	K_{p_0}^{(\varepsilon)}\lesssim\varepsilon^{\frac{1}{4}}R^{\frac{n-1}{2(n+1)}+} \quad \text{
	for }\, \, \, \varepsilon>R^{-\frac{4(n-1)}{n^2+6n-3}}.
	\label{N34}
\end{equation}

Assuming (\ref{N34}) and applying (\ref{N30}), $\rho=R$, gives
\begin{equation*}
	||M||_{p_1'\to p_1}\lesssim \varepsilon
	R+(\varepsilon^{\frac{1}{2}}R^{\frac{n-1}{n+1}+})^{1-\frac{2}{n(n-1)}}\Bigl(\frac{R}{\varepsilon}\Bigr)^{\frac{1}{n}}\lesssim \e R
\end{equation*}
provided that moreover
\begin{equation}
	\varepsilon\gtrsim R^{-\frac{2(n-1)}{n(n+1)}+}
	\label{N35}
\end{equation}
which supercedes the condition (\ref{N34}).  Thus, we have obtained the higher dimensional results in Theorem~\ref{theorem1.3} in the case of odd dimensions.

\noindent\underline{Case of $n$ even}

From (\ref{N16}) we have
\begin{equation*}
	\begin{split}
		K_{p_0}&\lesssim\varepsilon^{\frac{1}{4}}R^{\frac{n-1}{2(n+1)}+}+\varepsilon^{\frac{n}{4(n+1)}}
		R^{\frac{n-1}{2(n+1)}+}
	 \\ &\qquad\qquad
		+R^{\frac{n}{4(n+1)}}\max_{\rho_1>\rho_0}\rho_1^{-\frac{1}{2(n+1)}}K_{p_0}^{\frac{n}{2(n+1)}(1-\frac{2}{n(n-1)})}\Bigl(\frac{\rho_1}{\sqrt{\varepsilon
		R}}\Bigr)^{\frac{1}{2(n+1)}}\\
		&\lesssim\varepsilon^{\frac{n}{4(n+1)}}R^{\frac{n-1}{2(n+1)}+}+R^{\frac{n-1}{4(n+1)}}\varepsilon^{-\frac{1}{4(n+1)}}K_{p_0}^{\frac{n}{2(n+1)}-\frac{1}{(n-1)(n+1)}}\\
		&\lesssim\varepsilon^{\frac{n}{4(n+1)}}R^{\frac{n-1}{2(n+1)}+}+R^{\frac{n-1}{4(n+1)}}\varepsilon^{-\frac{1}{4(n+1)}}K_{p_0}^{\frac{n-2}{2(n-1)}}
	\end{split}
\end{equation*}
which implies that
\begin{equation}
	K_{p_0}^{(\varepsilon)}\lesssim \varepsilon^{\frac{n}{4(n+1)}}R^{\frac{n-1}{2(n+1)}},
	\text{ provided that } \, \, \varepsilon>R^{-\frac{2(n-1)}{n^2+2n-2}}.\label{N36}
\end{equation}

Assuming (\ref{N36}), application of (\ref{N30}) with $\rho=R$ gives
\begin{equation*}\begin{split}
	||M||_{p_1'\to p_1}&\lesssim\varepsilon
	R+(\varepsilon^{\frac{n}{2(n+1)}}R^{\frac{n-1}{n+1}+})^{1-\frac{2}{n(n-1)}}\Bigl(\frac{R}{\varepsilon}\Bigr)^{\frac{1}{n}}\\
	&\lesssim\varepsilon R,
\end{split}
\end{equation*}
meaning that we have obtained the results in Theorem~\ref{theorem1.3} for even dimensions.

This completes the proof of Theorem \ref{theorem1.3}. 
Note that we disregarded here the additional savings from non-trivial estimates on the
exponential sum (cf (\ref{4.1.41})), which will be small in above range for $\varepsilon$.

			\numberwithin{equation}{subsection}
			\numberwithin{theorem}{subsection}

\newsection{Improved  bounds for manifolds with nonpositive sectional curvatures}

To prove Theorem~\ref{theorem1.4}, in view of Theorem~\ref{theoremnecsuff}, our task is equivalent
to showing that if $(M,g)$ is a fixed compact manifold of dimension $n\ge3$ with nonpositive
sectional curvatures then
\begin{equation}\label{5.4}
\Bigl\| \, \sum_{|\la_j-\la|\le 1/\log \la}E_jf\, \Bigr\|_{L^{\frac{2n}{n-2}}(M)}\le C(\log \la)^{-1}\la\|f\|_{L^{\frac{2n}{n+2}}(M)}, 
\quad \la \gg 1.
\end{equation}
This is a special case of a recent unpublished estimate of Hassell and Tacey, which
is an $L^p$ variant of earlier bounds of B\'erard~\cite{Berard}.  For the sake of completeness, we shall
present a proof which is based on a slightly different interpolation scheme than the one employed by Hassell and Tacy.

Let us now present the proof of \eqref{5.4}.  If, as in \eqref{4.3}, $a\in {\mathcal S}(\R)$ is an even nonnegative function
satisfying $a(0)=1$ and $\text{supp } \Hat a\subset (-1,1)$, then by the proof of \eqref{4.2}, we would obtain \eqref{5.4}
if we could show that for a small fixed $\e_1$ to be determined later we have
\begin{equation}\label{5.5}
\Bigl\| \int_{-\infty}^\infty  \Hat a(t/\e_1\log \la)e^{it\la} \cos t\Gt f\, dt \Bigr\|_{L^{\frac{2n}{n-2}}(M)}
\le C\la \|f\|_{L^{\frac{2n}{n+2}}(M)}.
\end{equation}
To do this, as in \S~\ref{S4.1}, we choose an even function $b\in C_0^\infty(\R)$ satisfying
$$b(t)=1, \, \, |t|\le 1, \quad \text{and } \, \, b(t)=0, \, \, \, |t|\ge 2.$$  
If we then, exactly as we did in the beginning of  \S~\ref{S4.1}, split the operator
into two parts, $A_0$ and $A_1$ where we multiply the integrand by $b(t)$ and $(1-b(t))$, respectively,
then just as before $A_0$ satisfies the analog of \eqref{5.5} by virtue of Lemma~\ref{lemma2.2}.

Consequently, for the proof of \eqref{5.5}, we are left with proving that for appropriate $\e_1>0$
$$A_1f=\int_{-\infty}^\infty \bigl(1-b(t)\bigr) \, \Hat a(t/\e_1 \log \la) e^{it\la} \cos t\Gt f\, dt.$$
satisfies
\begin{equation}\label{5.6}\|A_1f\|_{\Lleft}\le C\la \|f\|_{\Lright}.
\end{equation}
We can even do a bit better than this, by interpolation, if we can show that given  $\delta_1>0$ we can
choose  $\e_1>0$  small but fixed so that we have
\begin{equation}\label{5.7}
\|A_1f\|_{L^{\frac{2(n+1)}{n-1}}(M)}\le C\la^{\frac{n-1}{n+1}} \log \la \|f\|_{L^{\frac{2(n+1)}{n+3}}(M)},
\end{equation}
and
\begin{equation}\label{5.8}
\|A_1 f\|_{L^\infty(M)}\le C\la^{\frac{n-1}2+\delta_1}\|f\|_{L^1(M)},
\end{equation}
for, together, if $\delta_1<\frac{n-1}2$,
they imply an improvement of \eqref{5.6} involving a smaller power of $\lambda$ in the right side.  This follows
from the M. Riesz interpolation since
$$\frac{n-2}{2n}=\frac{n-1}{2(n+1)} \theta, \quad \text{if } \, \, \, \theta = \frac{(n+1)(n-2)}{(n-1)n},$$
and for this value of $\theta$, we have
$$1=\frac{n-1}{n+1}\theta + (n-1)(1-\theta).$$

Repeating the proof of \eqref{4.6} shows that \eqref{5.7} follows immediately from Lemma~\ref{lemma2.2}
as $\sigma(p)$ there equals $\frac{n-1}{n+1}$ when $p=\frac{2(n+1)}{n-1}$.  Therefore, the proof of 
\eqref{5.4} would be complete if we could verify \eqref{5.8}.

To prove this inequality, we shall use the fact that, like for $\Tn$, because of our assumption of nonpositive
sectional curvatures, there is a Poisson-type formula that relates the kernel of $\cos t\Gt$ to a periodic
sum of wave kernels on the universal cover of $(M,g)$, which is $\Rn$.  If $p: \Rn\to M$ is a covering
map, we shall let $\tilde g$ denote the pullback of $g$ via $p$.  If $\Delta_{\tilde g}$ denotes the associated
Laplace-Beltrami operator on the universal cover $(\Rn, \tilde g)$, the formula we require is
\begin{equation}\label{5.9}
\bigl(\cos t\Gt\bigr)(x,y)=\sum_{\gamma \in \Gamma} \bigl(\cos t \Ct\bigr)(x,\gamma y), \quad x,y\in D,
\end{equation}
where $D\subset \Rn$ is a fixed fundamental domain, which we identify with $M$ via the covering map $p$,
and $\Gamma$ denotes the group of deck transformations for the covering.  The latter is the group of 
homeomorphisms $\gamma: \Rn \to \Rn$ for which $p=p\circ \gamma$, and $\gamma y$ denotes the image
of $y$ under $\gamma$.  Note then that if $(\tilde M,\tilde g)=(\Rn,\tilde g)$, then $M\simeq \tilde M/\Gamma$.

To use this formula let us first note that, by Huygens principle, 
$(\cos t\Ct)(x,y)=0$ if $d_{\tilde g}(x,y)>t$, with $d_{\tilde g}$ denoting the Riemannian distance
with respect to the metric $\tilde g$.  Consequently,
since $\Hat a(t)=0$ for $|t|\ge 1$,  in order to prove \eqref{5.8} it would be enough to show
that for $x,y\in D$ we have
\begin{multline}\label{5.10}
\sum_{\{\gamma \in \Gamma: \, d_{\tilde g}(x,\gamma y)\le \e_1 \log \la\}}
\left| \int_{-\infty}^\infty (1-b(t))\Hat a(t/\e_1\log \la) e^{i t\la}\bigl(\cos t\Ct\bigr)(x,\gamma y)\, dt
\right|
\\
\le C\la^{\frac{n-1}2 + \delta_1}.
\end{multline}
Since $(\Rn,\tilde g)$ has no conjugate points, following B\'erard~\cite{Berard}, we can use the Hadamard parametrix for large
times to prove estimates like \eqref{5.10}.  We no longer have uniform bounds on the amplitudes as we did in
\eqref{Had} and \eqref{2.27} for short times.  On the other hand, for $|t|\ge 1$, by writing the Fourier integral
terms in the  Hadamard parametrix in  an equivalent way to those in \cite{Berard} (cf. e.g., \cite{SoggeHang}, Remark 1.2.5),
we see that if $\tilde g$ is as above then there is a constant $c_0=c_0(\tilde g)>0$, which is independent of $T>1$ so that
\begin{equation}\label{5.11}
\bigl(\cos t\Ct\bigr)(x,y)=\sum_{\pm}\int_{\Rn} \alpha_{\pm}(t,x,y,|\xi|)
e^{i\kappa(x,y)\cdot \xi} e^{\pm it|\xi|}\, d\xi
+R(t,x,y),
\end{equation}
where, as before, $\kappa(x,y)$ denotes the geodesic normal coordinates (now with respect to $\tilde g$) of $x$ about $y$,
but now \eqref{2.27} has to be replaced by
\begin{multline}\label{5.12}
\left| \, \frac{d^j}{dt^j}\frac{d^k}{dr^k}\alpha_\pm(t,x,y,r)\, \right|
\le A_{jk}e^{c_0T} (1+r)^{-k}, 
\\
\text{if } \, \, 1\le |t|\le T, \, \, r>0, \, \, 
\text{and } \, \, j,k\in \{0,1,2,\dots\}.
\end{multline}
The remainder term can be taken to be continuous, but it also satisfies bounds that become exponentially worse with time:
\begin{equation}\label{5.1.12}
|R(t,x,y)|\le Ae^{c_0T}, \quad \text{if } \, 1\le |t|\le T.
\end{equation}

To use these we shall require the following simple stationary phase lemma.

\begin{lemma}\label{lemma5.1}  Suppose that $\alpha(t,r)\in C^\infty(\R\times \R_+)$ satisfies
$$\alpha(t,r)=0 \quad \text{if } \, \, |t|\notin [1,T],$$
and for every $j=0,1,2,\dots$ and $k=0,1,2,\dots$
\begin{equation}\label{5.13}
\left| \, \frac{d^j}{dt^j}\frac{d^k}{dr^k}\alpha(t,r)\right|\le A_{jk}e^{cT}(1+r)^{-k},
\end{equation}
for a fixed constant $c>0$.  Then there is a constant $B$ depending only on $n$ and the size of finitely many of 
the constants in \eqref{5.13} so that
$$\left|  \, \int_{-\infty}^\infty \int_{\Rn}
\alpha(t,|\xi|)e^{it\la\pm it|\xi|}e^{iv\cdot \xi}\, d\xi dt \,\right|
\le BTe^{cT}\la^{\frac{n-1}2}, \quad \la, T>1.$$
\end{lemma}

\begin{proof}  If $|v|\le 1/2$, then one can integrate by parts in the $\xi$-variable to see that better bounds hold
where in the right side $\lambda^{\frac{n-1}2}$ is replaced by one.  The result for $|v|\ge 1/2$ follows
from \eqref{surface} and \eqref{surface2} after noting that if $\Hat \alpha(\tau,|\xi|)$ denotes the partial
Fourier transform in the $t$-variable then our assumptions \eqref{5.13} and an integration by parts
argument imply that
$$\int_{\Rn}|\Hat \alpha(\la\pm |\xi|,|\xi|)| \, |\xi|^{-\frac{n-1}2}\, d\xi \le BTe^{cT}\la^{\frac{n-1}2}.$$
\end{proof}

We can now finish the proof of \eqref{5.10}.  If we use the lemma along with \eqref{5.11}-\eqref{5.1.12}, 
we conclude that each term in the sum in the left side of \eqref{5.10} is bounded by
\begin{equation}\label{5.14}
BTe^{c_0T}\la^{\frac{n-1}2}, \quad T=\e_1\log \la,
\end{equation}
for some constant $B$ when $\la \gg 1$.  Furthermore, as noted by B\'erard~\cite{Berard} classical volume growth estimates
 imply that there are $O(e^{c_1T})$ nonzero terms in the sum in \eqref{5.10} for a fixed constant $c_1=c_1(\tilde g)$.  Thus,
 after possibly increasing the constant $c_0$ in \eqref{5.14}, the whole sum is bounded by this quantity.  Since given
 $\delta_1>0$ we can choose a fixed $\e_1>0$, $\e_1=\e_1(c_0,\delta_1)$ so that 
 the quantity in \eqref{5.14} is $O(\la^{\frac{n-1}2 +\delta_1})$ as $\la\to +\infty$, we obtain \eqref{5.10}, and hence \eqref{5.4}. \qed

\end{document}